\documentclass{amsart}
\usepackage{amsmath, amsthm, amssymb}
\usepackage{amsmath, amssymb}
 \usepackage{amsmath,amscd,amsthm}
\usepackage{mathrsfs}
\usepackage{tikz}
\input xypic
\xyoption{all}

\theoremstyle{plain} 
\newtheorem{theorem}[subsection]{Theorem}
\newtheorem{proposition}[subsection]{Proposition}
\newtheorem{lemma}[subsection]{Lemma}

\theoremstyle{definition}

\theoremstyle{remark}

\newtheorem*{remark*}{Remark}

\numberwithin{equation}{subsection}
\newcommand{\set}[1]{\{#1\}}

\newcommand{\mset}[1]{\set{\,#1\,}}

\newcommand{\pair}[1]{\langle #1\rangle}

\newcommand{\mpair}[1]{\pair{\,#1\,}}

\DeclareMathOperator{\sgn}{{\mathrm{sgn}}}
\DeclareMathOperator{\cone}{{\mathrm{cone}}}
\DeclareMathOperator{\linspan}{{\mathrm{span}}}
\DeclareMathOperator{\Hom}{{\mathrm{Hom}}}
\newcommand{\CP}{\mathcal{P}}

\begin{document}

\title[Minimal infinite reduced words of affine Weyl groups]{Reduced expression of minimal infinite reduced words of affine Weyl groups}

\author{Weijia Wang}
\address{Shing-Tung Yau Center
\\ Southeast University \\
Nanjing, Jiangsu, 210000 \\ China}
\email{wangweij5@mail.sysu.edu.cn}

\begin{abstract}
For an infinite Coxeter system, one can extend the weak right order to the set of infinite reduced words. This is called limit weak order.
In [Transformation Groups 18(1), 2013, 179-231], Lam and Pylyavskyy showed that for affine Weyl groups of type $\widetilde{A}_n$  minimal infinite reduced words under the limit weak order are precisely those infinite Coxeter elements and asked the question of  characterization, in terms of infinite reduced words, of the minimal elements
of the limit weak order for other affine types. In this paper we answer this question by characterizing the minimal infinite reduced words for other irreducible affine Weyl groups by one of their reduced expressions.
\end{abstract}
\maketitle
\section{Introduction}\label{bkg}

Enumeration of the reduced expressions of an element in a Coxeter group is an important theme in the study of the combinatorics of Coxeter groups. For an infinite Coxeter group, one can consider the infinite reduced words (with respect to the simple reflections), which are studied for their interplay with the Tits boundary, twisted weak order, biclosed sets and dynamic system etc. Finding the reduced expressions of an infinite reduced word is a similarly natural and interesting combinatorics question. However little is known about the reduced expressions of infinite reduced words. In \cite{Lam2}, Lam and Pylyavskyy showed that for affine Weyl groups of type $\widetilde{A}_n$  minimal infinite reduced words under the limit weak order are precisely those infinite Coxeter elements, i.e. they have a reduced expression of the form $(s_{i_1}\cdots s_{i_{n+1}})^{\infty}$ where $\widetilde{S}=\{s_{i_1},\cdots, s_{i_{n+1}}\}$ ($\widetilde{S}$ is the set of simple reflections for that affine Weyl group). Among infinite reduced words, minimal ones are naturally first objects to investigate.  In this paper we obtain similar results for other irreducible affine Weyl groups, i.e. we find a specific reduced expression for all minimal infinite reduced words. This is open problem 1 in \cite{Lam2}.

The paper is organized as follow. In section \ref{prelim} some background on Coxeter groups and affine Weyl groups is reviewed. In section
\ref{genCox} we prove the existence of minimal infinite reduced words for all finite rank Coxeter groups. We also
determine the number of minimal infinite reduced words for an irreducible affine Weyl group. In section \ref{SectMain} we state and prove our main theorem, i.e. the description
of a specific reduced expression for the minimal infinite reduced words of all irreducible affine Weyl groups from type $\widetilde{B}$ to $\widetilde{G}$. The type specific verification
in this proof is addressed in section \ref{translation}. In section \ref{infiniteCox} we investigate which of these minimal infinite reduced words are infinite Coxeter elements. By introducing a Weyl group action on the set of  Coxeter elements of the corresponding affine Weyl group, we give a satisfactory answer of this question for type $\widetilde{C_n}, \widetilde{F}_4$ and $\widetilde{G}_2$. We also answer a question in \cite{Lam2} regarding the full commutativity of the minimal infinite reduced words.

\section{Preliminaries}\label{prelim}
Let $(W,S)$ be a Coxeter system such that $S$ is finite. See \cite{bjornerbrenti} and \cite{Hum} for the basics of the Coxeter groups and their root systems. We review here the notions and properties of Coxeter groups that are beyond the scope of these books and are needed in this paper.

For $w\in W$, the inversion set of $w$ is defined to be $\{\alpha\in \Phi^+|w^{-1}(\alpha)\in \Phi^-\}$ and is denoted by $\Phi_w$.

Assume that $W$ is infinite. An infinite  sequence $s_1s_2s_3\cdots, s_i\in S$ is called an infinite reduced word of $W$ if $s_1s_2\cdots s_j$ is  reduced  for any $j\geq 1$. The inversion set of an infinite reduced word $s_1s_2\cdots$, denoted by $\Phi_{s_1s_2\cdots}$, is the union $\cup_{i=1}^{\infty}\Phi_{s_1s_2\cdots s_i}$. Two infinite reduced words are considered equal if their inversion sets are equal. The set of infinite reduced words of $(W,S)$ will be denoted by $W_l$. We denote $W\cup W_l$ by $\overline{W}$. Let $u\in W, v\in W_l$, one can define the multiplication $uv$ as in Definition 2.6 in \cite{wang}.  An element $w\in W$ is called straight if $l(w^n)=|nl(w)|.$ It is proved in \cite{speyer} that for an irreducible and infinite Coxeter group $W$ any Coxeter element is straight. Given a reduced expression $\underline{w}$ of a straight element $w$, the infinite reduced word $\underline{w}\underline{w}\underline{w}\cdots$  is independent of the choice of $\underline{w}$ and we shall denote this infinite reduced word by $w^{\infty}$.

A subset $\Gamma$ of the set of roots $\Phi$ is said to be closed if for any $\alpha,\beta\in \Gamma$ and $k_1\alpha+k_2\beta\in\Phi, k_1,k_2\in \mathbb{R}_{\geq 0}$ one has that  $k_1\alpha+k_2\beta\in\Gamma$. A set $B\subset \Gamma$ such that both $B$ and $\Gamma\backslash B$ are  closed  is called a  biclosed set in $\Gamma$. The inversion set of an element in the Coxeter group or an infinite reduced word is biclosed in $\Phi^+$, the set of positive roots. There exists a $W$-action on the set of all biclosed sets in $\Phi^+$ given by $w\cdot B:=(N(w)\backslash w(-B))\cup (w(B)\backslash
(-N(w)))$. For this action, see \cite{DyerWeakOrder}. In particular $u\cdot N(v)=N(uv)$ for $u\in W$ and $v\in \overline{W}$. Here the natural action of $W$ on a set $A\subset \Phi$ is denoted by $w(A)$ (i.e. $w(A)=\{w(\alpha)|\alpha\in A\}$.)

Let $W$ be an irreducible finite Weyl group with the crystallographic root system $\Phi$ contained in the Euclidean space $V$. Let $\Phi^+$ be the chosen standard positive system of $\Phi$ and let $\Pi$ be the simple system of $\Phi^+.$ For these notions, see Chapter 1 of \cite{Hum}. The root system of an (irreducible) affine Weyl group $\widetilde{W}$ (corresponding to $W$) can be constructed in the following way.

Let $\delta$ be an indeterminate.
Construct a $\mathbb{R}-$vector space $V'=V\oplus\mathbb{R}\delta$ and extend
the inner product on $V$ to $V'$ by requiring $(\delta,v)=0$ for any $v\in V'$.
If $\alpha\in \Phi^+$, define
$\widehat{\{\alpha\}}=\{\alpha+n\delta|n\in \mathbb{Z}_{\geq 0}\}\subset
V'$. If $\alpha\in \Phi^-$, define
$\widehat{\{\alpha\}}=\{\alpha+(n+1)\delta|n\in \mathbb{Z}_{\geq
0}\}\subset V'$.
For a set $\Lambda\subset \Phi$, define
$\widehat{\Lambda}=\bigcup_{\alpha\in\Lambda}\widehat{\{\alpha\}}\subset
V'$.

Then the set of roots of the affine Weyl group $\widetilde{W}$, denoted by $\widetilde{\Phi}$, is $\widehat{\Phi}\uplus-\widehat{\Phi}$.
The set of positive roots  (resp. the set of negative roots) is $\widehat{\Phi}$ (resp. $-\widehat{\Phi}$). The set of simple roots is $\{\alpha|\alpha\in \Delta\}\cup\{\delta-\rho\}$ where $\rho$ is the highest root in $\Phi^+$. Let $\alpha$ be a root in $\widetilde{\Phi}$. The reflection in $\alpha$, denoted by $s_{\alpha}$, is the map $V'\rightarrow V': v\mapsto v-2\frac{(v,\alpha)}{(\alpha,\alpha)}\alpha.$
Then the (irreducible) affine Weyl group $\widetilde{W}$ is generated by
$s_{\alpha},\alpha\in \widetilde{\Phi}$. It is known to be a Coxeter group with the simple reflections being the reflections in the simple roots of $\widehat{\Phi}$.  For $v\in V$,  define the $\mathbb{R}-$linear map $t_v: V'\rightarrow V'$   by $t_v(u)=u+(u,v)\delta.$ For $\alpha\in \Phi,$ define the coroot $\alpha^{\vee}=2\frac{\alpha}{(\alpha,\alpha)}$.  Let $T$  be the free Abelian group generated by $\{t_{\gamma^{\vee}}|\gamma\in \Delta\}$. Then one has that $\widetilde{W}=W\ltimes T.$ It is known that $t_{\alpha}$ is straight for a coroot $\alpha.$ We call an element $t_{\alpha}$ a ``translation'' where $\alpha$ is in the coroot lattice (the lattice spanned by the coroot).  Let $\pi$ be the canonical projection from $\widetilde{W}$ to $W$.

There is a classification of biclosed sets in the set of positive roots which are the inversion sets of infinite reduced words for affine Weyl groups in \cite{wang}, which is based on previous work in \cite{DyerReflOrder}. We recall it here. For $\Delta'\subset \Phi$, denote by $\Phi_{\Delta'}$ the root subsystem generated by $\Delta'$. It is shown in \cite{DyerReflOrder} and \cite{biclosedphi} that the biclosed sets in $\Phi$ are in the form $(\Psi^+\backslash \Phi_{\Delta_1})\cup \Phi_{\Delta_2}$ where $\Psi^+$ is a positive system of $\Phi$ and $\Delta_1,\Delta_2$ are two orthogonal subsets (i.e. $(\alpha,\beta)=0$ for any $\alpha\in \Delta_1,\beta\in \Delta_2$) of the simple system of $\Psi^+.$ We denote the set $(\Psi^+\backslash \Phi_{\Delta_1})\cup \Phi_{\Delta_2}$ by $\Psi^+_{\Delta_1,\Delta_2}$.

Any biclosed set in $\widetilde{\Phi}^+(=\widehat{\Phi})$ is of the form $w\cdot \widehat{\Psi^+_{\Delta_1,\Delta_2}}$ for some $\Psi^+,\Delta_1,\Delta_2$ and $w\in U$ where $U$ is the reflection subgroup generated by $\widehat{\Phi_{\Delta_1\cup \Delta_2}}$. The biclosed sets which are inversion sets of infinite reduced words are precisely those of the form $w\cdot \widehat{\Psi^+_{\Delta_1,\emptyset}}, w\in U,\Delta_1\subsetneq \Delta$ where $U$ is the reflection subgroup generated by $\widehat{\Phi_{\Delta_1}}$. In particular $w\cdot \widehat{\Psi^+_{\Delta_1,\emptyset}}$ equals $\widehat{\Psi^+_{\Delta_1,\emptyset}}\cup B$ where $B$ is a finite biclosed set in $\widehat{\Phi_{\Delta_1}}$ and the union is disjoint.

The extended weak order on $\overline{W}$ (denoted $(\overline{W},\leq)$) is a partial order defined as: $u\leq v$ if and only if $\Phi_u\subset \Phi_v.$
The extended weak order restricted on $W$ (denoted $(W,\leq)$) is called weak (right) order and the extended weak order restricted on $W_l$ (denoted $(W_l,\leq)$) is called limit weak order. If $u\leq v$ for $u\in W, v\in \overline{W}$, $u$ is called a (left) prefix of $v$.

\section{Existence of minimal infinite reduced words}\label{genCox}

In this section we address the problem of the existence of  minimal infinite reduced words and discuss the cardinality of the set of minimal infinite reduced words.

\begin{theorem}
Suppose that $S$ is finite. The limit weak order on $W_l$ has minimal elements.
\end{theorem}

\begin{proof}
We first show that under the limit weak order any decreasing chain of $W_l$ has a greatest lower bound in $W_l$.

Since $(\overline{W},\leq)$  is a complete meet semi-lattice by Theorem 2.9 in \cite{wang}, such chain must have a greatest lower bound in $\overline{W}$.
We show that such a greatest lower bound cannot be $e$. Assume to the contrary the greatest lower bound is the identity element $e$. Suppose that $\{\alpha_1,\alpha_2,\cdots,\alpha_r\}$ is the set of simple roots. Then $\alpha_1$ cannot be in the inversion set of all elements in this chain (otherwise the greatest lower bound would be greater than or equal to $s_{\alpha_1}$). So we can find an element $u$ in this chain whose inversion set does not contain $\alpha_1.$ Then we consider the elements in the chain below $u$. They form another decreasing sub-chain (whose meet is also $e$ and $s_{\alpha_1}$ is not a prefix of any element in this sub-chain). Proceed in this manner we will get a sub-chain such that any simple reflection is not the prefix of any element of it.  But this is a sub-chain of infinite reduced words. This is a contradiction.

Now we show that such a greatest lower bound cannot be some $w\in W$ either. If so, we left multiply all elements in the chain by $w^{-1}$ and the new chain will have the greatest lower bound $e.$

This implies that the greatest lower bound of a chain of infinite reduced words must  lie in $W_l$. By applying Zorn's Lemma to the opposite poset of the limit weak order on $W_l$, we see the conclusion.
\end{proof}

The minimal elements in $(W_l,\leq)$ will be called minimal infinite reduced words.

\begin{remark*}
If $S$ is infinite, it could happen that $W_l$ has no minimal element.
Consider the weak direct product of $W_i,i=1,2,\cdots$ where each $W_i$ is of type $A_1$ (i.e. $\mathbb{Z}_2$).
Then any infinite reduced word is of the form $s_1s_2\cdots, s_i\in S, s_i\neq s_j$ if $i\neq j$.
One can easily see that any sub-infinite reduced word of it is less than it under the limit weak order.
\end{remark*}

\begin{theorem}\label{thmaffineminimal}
(1) For an irreducible affine Weyl group $\widetilde{W}$, the minimal infinite reduced words are those infinite reduced words whose inversion set is of the form $\widehat{\Psi^+_{\Delta_1,\emptyset}}$ with $|\Delta_1|=|\Delta|-1$ where $\Delta$ is the simple system of $\Psi^+.$

(2) The number of minimal infinite reduced words for $\widetilde{W}$ is $|W|(\sum_{H} \frac{1}{|H|})$ where the sum runs over the set of maximal proper standard parabolic subgroups.
\end{theorem}

\begin{proof}
(1) Let $u$ be a minimal infinite reduced word. By section \ref{prelim} we can assume that $\Phi_u=w\cdot \widehat{\Psi^+_{\Delta_1,\emptyset}}$ where $w$ is an element in the reflection subgroup generated by $\widehat{\Phi_{\Delta_1}}$. Then $w$ has to be the identity since otherwise $w\cdot \widehat{\Psi^+_{\Delta_1,\emptyset}}$ equals $\widehat{\Psi^+_{\Delta_1,\emptyset}}\cup B$ where $B$ is a finite biclosed set in $\widehat{\Phi_{\Delta_1}}$ (See Theorem 1.3(2) in \cite{wang}). Hence $\Phi_u=\widehat{\Psi^+_{\Delta_1,\emptyset}}$. Since $\Delta_1\subsetneq \Delta$, the minimality forces that $|\Delta_1|=|\Delta|-1.$

Conversely let $D=\widehat{\Psi^+_{\Delta_1,\emptyset}}$ with $|\Delta_1|=|\Delta|-1$ where $\Delta$ is the simple system of $\Psi^+.$ Then by section \ref{prelim}, $D$ is the inversion set of an infinite reduced word $u$.  Let $u'$ be another infinite reduced word such that $u'\leq u$, by section \ref{prelim}, $\Phi_{u'}=\widehat{\Xi^+_{\Pi_1,\emptyset}}\cup B'$ where $B'$ is a finite biclosed set in $\widehat{\Phi_{\Pi_1}}$. Since $\Phi_{u'}\subset \Phi_u$,
this forces $\Xi^+_{\Pi_1,\emptyset}\subset \Psi^+_{\Delta,\emptyset}$. Consider the poset of biclosed sets of $\Phi$ (under containment). Biclosed sets of the form  $\Psi^+_{\Delta_1,\emptyset}$ with $|\Delta_1|=|\Delta|-1$ where $\Delta$ is the simple system of $\Psi^+$ are precisely the atoms in this poset (see \cite{DyerReflOrder} or \cite{biclosedphi}), i.e. only empty (biclosed) set ($=\Psi^+_{\Delta,\emptyset}$) is contained in them. So $\Xi^+_{\Pi_1,\emptyset}=\Psi^+_{\Delta_1,\emptyset}$ or $\Xi^+_{\Pi_1,\emptyset}=\emptyset$. The latter is impossible since $\Phi_{u'}$ is infinite. Therefore $\Xi^+_{\Pi_1,\emptyset}=\Psi^+_{\Delta_1,\emptyset}$ and $B'=\emptyset$. So $u'=u.$

(2) By (1) the minimal infinite reduced words are in bijection with the biclosed sets (in $\Phi$) of the form $\Psi^+_{\Delta_1,\emptyset}$ with $|\Delta_1|=|\Delta|-1$. Therefore we count the number of the biclosed sets of that form. To do this, one considers the natural action of the finite Weyl group $W$ on these biclosed sets. We take $\Psi^+=\Phi^+$, the standard positive system. By Theorem 1.15 of \cite{DyerReflOrder} every biclosed set of the form $\Psi^+_{\Delta_1,\emptyset}$ with $|\Delta_1|=|\Delta|-1$ is conjugate to some unique $\Phi^+_{M,\emptyset}$ with $|M|=\mathrm{rank}(W)-1.$ Now the assertion follows from the fact that the stabilizer of $\Phi^+_{M,\emptyset}$ is precisely $W_M,$ the standard parabolic subgroup generated by $s_{\alpha},\alpha\in M.$
\end{proof}

A Coxeter system $(W,S)$ is said to be word-hyperbolic if and only if there exists no $R\subset S$ which satisfies one of the following conditions: (1) $W_R$ is affine and $(W_R,R)$ has rank at least 3 (as a Coxeter system). (2) $(W_R, R)=(W_{R_1}\times W_{R_2}, R_1\cup R_2)$ with both $W_{R_1}$ and $W_{R_2}$ both infinite.

 \begin{proposition}\label{infinitewordhyperbolic}
 Let $(W,S)$ be an  infinite non-affine word-hyperbolic Coxeter system. Then $(W_l,\leq)$ admits infinitely many minimal elements.
 \end{proposition}

 \begin{proof}
 By \cite{Lam} Theorem 3, two infinite reduced words are comparable if and only if they are in the same block of $(W_l,\leq)$. (Two infinite reduced words are in the same block if and only if the symmetric difference of their inversion sets is finite.) Furthermore the limit weak order limited to a block is isomorphic to the weak order of a finite Coxeter group. Therefore each block has a unique maximal element (which is also a maximal element of $W_l$) and a unique minimal element (which is also a minimal element of $W_l$). By \cite{wang} Theorem 3.16, $\overline{W}$ has infinitely many maximal elements. Therefore $(W_l,\leq)$ has infinitely many blocks. Therefore it has infinitely many minimal elements.
 \end{proof}

\section{the theorem}\label{SectMain}

Let $(U,S)$ be a finite Coxeter system. We denote by $w_0$ the longest element of $U$. Let $V$ be a standard parabolic subgroup of $U$. We denote by $w_0^{V}$ the longest element of $V$. For $J\subset S$, we denote by $W_J$ the parabolic subgroup generated by $J$. Let $\Pi$ be the set of simple roots of the root system of $U$. For $\Gamma\subset \Pi$ define
$W_{\Gamma}:=W_{J_{\Gamma}}$ where $J_{\Gamma}=\{s_{\alpha}|\,\alpha\in \Gamma\}$.

Let $W$ be a finite irreducible Weyl group of type $X_n, X=B,C,D,E,F,G$. We denote the simple roots by $\beta_1, \beta_2, \cdots, \beta_n.$ The numbering of these simple roots comes from the following Dynkin diagram:

Type $B_n$

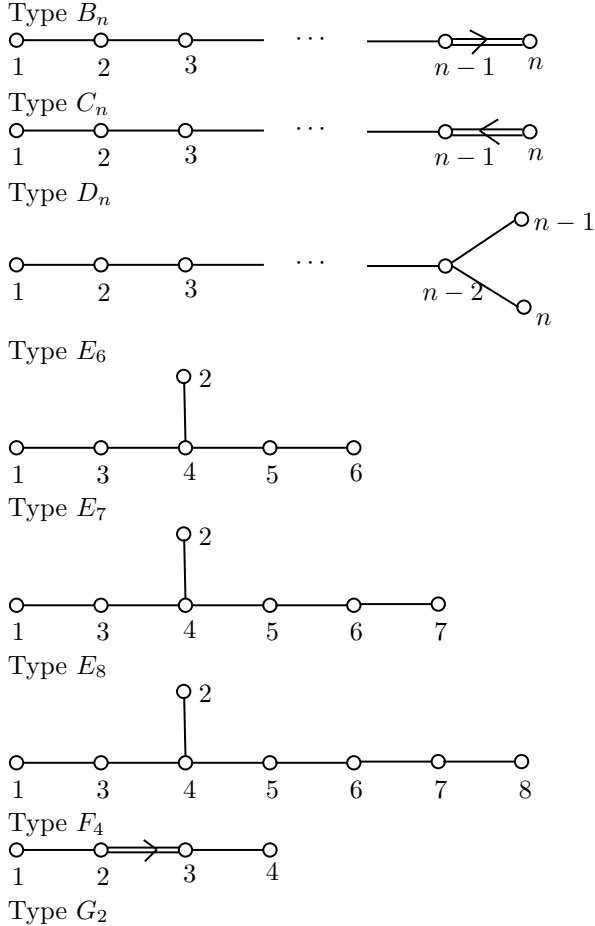
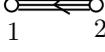
\tikzset{every picture/.style={line width=0.75pt}} 



\begin{theorem}\label{main}
(1) For $W$ of type $B_n$, an infinite reduced word of $\widetilde{W}$ is minimal if and only if it has a reduced expression of the forms
$$(\underline{v}s_{\delta-\rho}((s_{\beta_2}s_{\beta_1}s_{\beta_3}s_{\beta_2})(s_{\beta_4}s_{\beta_3}s_{\beta_5}s_{\beta_4})\cdots (s_{\beta_{i-2}}s_{\beta_{i-3}}s_{\beta_{i-1}}s_{\beta_{i-2}})s_{\delta-\rho})$$
$$\cdots((s_{\beta_2}s_{\beta_1}s_{\beta_3}s_{\beta_2})(s_{\beta_4}s_{\beta_3}s_{\beta_5}s_{\beta_4})s_{\delta-\rho})((s_{\beta_2}s_{\beta_1}s_{\beta_3}s_{\beta_2})s_{\delta-\rho})\underline{u})^{\infty}$$
where $\underline{u}\underline{v}$ is a reduced expression of $w_0^{W_{S\backslash \{s_{\beta_i}\}}}w_0, 2\leq i\leq n$ and $i$ is even,
$$(\underline{v}(s_{\delta-\rho}(s_{\beta_2}s_{\beta_3}\cdots s_{\beta_n}s_{\beta_{n-1}}s_{\beta_{n-2}}\cdots s_{\beta_i}s_{\beta_1}s_{\beta_2}\cdots s_{\beta_n}s_{\beta_{n-1}}s_{\beta_{n-2}}\cdots s_{\beta_i}))^{\frac{i-1}{2}}$$
$$(s_{\delta-\rho}s_{\beta_2}s_{\beta_3}\cdots s_{\beta_n}s_{\beta_{n-1}}s_{\beta_{n-2}}\cdots s_{\beta_i}s_{\beta_1}s_{\beta_2}\cdots s_{\beta_{i-1}})$$
$$(s_{\delta-\rho}(s_{\beta_2}s_{\beta_1}s_{\beta_3}s_{\beta_2})(s_{\beta_4}s_{\beta_3}s_{\beta_5}s_{\beta_4})\cdots (s_{\beta_{i-3}}s_{\beta_{i-4}}s_{\beta_{i-2}}s_{\beta_{i-3}}))\cdots$$
$$(s_{\delta-\rho}(s_{\beta_2}s_{\beta_1}s_{\beta_3}s_{\beta_2})(s_{\beta_4}s_{\beta_3}s_{\beta_5}s_{\beta_4}))(s_{\delta-\rho}(s_{\beta_2}s_{\beta_1}s_{\beta_3}s_{\beta_2}))s_{\delta-\rho}\underline{u})^{\infty}$$
where $\underline{u}\underline{v}$ is a reduced expression of $w_0^{W_{S\backslash \{s_{\beta_i}\}}}w_0, 3\leq i\leq n$ and $i$ is odd,
$$(\underline{v}s_{\delta-\rho}s_{\beta_2}s_{\beta_3}\cdots s_{\beta_n}\cdots s_{\beta_3}s_{\beta_2}s_{\delta-\rho}\underline{u})^{\infty}$$
where $\underline{u}\underline{v}$ is a reduced expression of $w_0^{W_{S\backslash \{s_{\beta_1}\}}}w_0$.

(2) For $W$ of type $C_n$, an infinite reduced word of $\widetilde{W}$ is minimal if and only if it has a reduced expression of the form
$$(\underline{v}s_{\delta-\rho}(s_{\beta_1}s_{\delta-\rho})(s_{\beta_2}s_{\beta_1}s_{\delta-\rho})\cdots(s_{\beta_{i-1}}s_{\beta_{i-2}}\cdots s_{\beta_1}s_{\delta-\rho})\underline{u})^{\infty}$$
where $\underline{u}\underline{v}$ is a reduced expression of $w_0^{W_{S\backslash \{s_{\beta_i}\}}}w_0$ for $1\leq  i\leq n$.

(3) For $W$ of type $D_n$, an infinite reduced word of $\widetilde{W}$ is minimal if and only if it has a reduced expression of the forms
$$(\underline{v}s_{\delta-\rho}s_{\beta_2}s_{\beta_3}\cdots s_{\beta_{n-2}}s_{\beta_{n-1}}s_{\beta_{n}}s_{\beta_{n-2}}\cdots s_{\beta_3}s_{\beta_2}s_{\delta-\rho}\underline{u})^{\infty}$$
where $\underline{u}\underline{v}$ is a reduced expression of $w_0^{W_{S\backslash \{s_{\beta_1}\}}}w_0$.
$$(\underline{v}s_{\delta-\rho}((s_{\beta_2}s_{\beta_1}s_{\beta_3}s_{\beta_2})(s_{\beta_4}s_{\beta_3}s_{\beta_5}s_{\beta_4})\cdots (s_{\beta_{i-2}}s_{\beta_{i-3}}s_{\beta_{i-1}}s_{\beta_{i-2}})s_{\delta-\rho})$$
$$\cdots((s_{\beta_2}s_{\beta_1}s_{\beta_3}s_{\beta_2})(s_{\beta_4}s_{\beta_3}s_{\beta_5}s_{\beta_4})s_{\delta-\rho})((s_{\beta_2}s_{\beta_1}s_{\beta_3}s_{\beta_2})s_{\delta-\rho})\underline{u})^{\infty}$$
where $\underline{u}\underline{v}$ is a reduced expression of $w_0^{W_{S\backslash \{s_{\beta_i}\}}}w_0$ where $i$ is even and  $i\leq n-2$
$$(\underline{v}s_{\delta-\rho}(s_{\beta_2}s_{\beta_3}\cdots s_{\beta_n}s_{\beta_1}s_{\beta_2}\cdots s_{\beta_n}(s_{\beta_{n-3}}s_{\beta_{n-2}})(s_{\beta_{n-4}}s_{\beta_{n-3}})\cdots (s_{\beta_{i-1}}s_{\beta_i})s_{\delta-\rho})^{\frac{i-1}{2}}$$
$$(s_{\beta_2}s_{\beta_3}\cdots s_{\beta_n}s_{\beta_{n-2}}s_{\beta_{n-3}}\cdots s_{\beta_i}s_{\beta_1}s_{\beta_2}\cdots s_{\beta_{i-1}}s_{\delta-\rho})$$
$$((s_{\beta_2}s_{\beta_1}s_{\beta_3}s_{\beta_2})(s_{\beta_4}s_{\beta_3}s_{\beta_5}s_{\beta_4})\cdots (s_{\beta_{i-3}}s_{\beta_{i-4}}s_{\beta_{i-2}}s_{\beta_{i-3}})s_{\delta-\rho})$$
$$\cdots((s_{\beta_2}s_{\beta_1}s_{\beta_3}s_{\beta_2})(s_{\beta_4}s_{\beta_3}s_{\beta_5}s_{\beta_4})s_{\delta-\rho})((s_{\beta_2}s_{\beta_1}s_{\beta_3}s_{\beta_2})s_{\delta-\rho})\underline{u})^{\infty}$$
where $\underline{u}\underline{v}$ is a reduced expression of $w_0^{W_{S\backslash \{s_{\beta_i}\}}}w_0$ where $i$ is odd and  $3\leq i\leq n-2$
$$(\underline{v}s_{\delta-\rho}(s_{\beta_2}s_{\beta_3}\cdots s_{\beta_{n-1}}s_{\beta_1}s_{\beta_2}\cdots s_{\beta_{n-2}}s_{\beta_n}s_{\delta-\rho})^{\frac{n+1}{2}}$$
$$((s_{\beta_2}s_{\beta_1}s_{\beta_3}s_{\beta_2})(s_{\beta_4}s_{\beta_3}s_{\beta_5}s_{\beta_4})\cdots (s_{\beta_{n-3}}s_{\beta_{n-4}}s_{\beta_{n-2}}s_{\beta_{n-3}})s_{\delta-\rho})$$
$$\cdots((s_{\beta_2}s_{\beta_1}s_{\beta_3}s_{\beta_2})(s_{\beta_4}s_{\beta_3}s_{\beta_5}s_{\beta_4})s_{\delta-\rho})((s_{\beta_2}s_{\beta_1}s_{\beta_3}s_{\beta_2})s_{\delta-\rho})\underline{u})^{\infty}$$
where $\underline{u}\underline{v}$ is a reduced expression of $w_0^{W_{S\backslash \{s_{\beta_{n-1}}\}}}w_0$ where $n$ is odd,
$$(\underline{v}s_{\delta-\rho}(s_{\beta_2}s_{\beta_3}\cdots s_{\beta_{n-2}}s_{\beta_{n}}s_{\beta_1}s_{\beta_2}\cdots s_{\beta_{n-2}}s_{\beta_{n-1}}s_{\delta-\rho})^{\frac{n+1}{2}}$$
$$((s_{\beta_2}s_{\beta_1}s_{\beta_3}s_{\beta_2})(s_{\beta_4}s_{\beta_3}s_{\beta_5}s_{\beta_4})\cdots (s_{\beta_{n-3}}s_{\beta_{n-4}}s_{\beta_{n-2}}s_{\beta_{n-3}})s_{\delta-\rho})$$
$$\cdots((s_{\beta_2}s_{\beta_1}s_{\beta_3}s_{\beta_2})(s_{\beta_4}s_{\beta_3}s_{\beta_5}s_{\beta_4})s_{\delta-\rho})((s_{\beta_2}s_{\beta_1}s_{\beta_3}s_{\beta_2})s_{\delta-\rho})\underline{u})^{\infty}$$
where $\underline{u}\underline{v}$ is a reduced expression of $w_0^{W_{S\backslash \{s_{\beta_{n}}\}}}w_0$ where $n$ is odd,
$$(\underline{v}s_{\delta-\rho}s_{\beta_2}\cdots s_{\beta_{n-2}}s_{\beta_n}s_{\beta_1}s_{\beta_2}\cdots s_{\beta_{n-2}}s_{\delta-\rho}$$
$$((s_{\beta_2}s_{\beta_1}s_{\beta_3}s_{\beta_2})(s_{\beta_4}s_{\beta_3}s_{\beta_5}s_{\beta_4})\cdots (s_{\beta_{n-4}}s_{\beta_{n-5}}s_{\beta_{n-3}}s_{\beta_{n-4}})s_{\delta-\rho})$$
$$\cdots((s_{\beta_2}s_{\beta_1}s_{\beta_3}s_{\beta_2})(s_{\beta_4}s_{\beta_3}s_{\beta_5}s_{\beta_4})s_{\delta-\rho})((s_{\beta_2}s_{\beta_1}s_{\beta_3}s_{\beta_2})s_{\delta-\rho})\underline{u})^{\infty}$$
where $\underline{u}\underline{v}$ is a reduced expression of $w_0^{W_{S\backslash \{s_{\beta_{n-1}}\}}}w_0$ where $n$ is even,
$$(\underline{v}s_{\delta-\rho}s_{\beta_2}\cdots s_{\beta_{n-2}}s_{\beta_{n-1}}s_{\beta_1}s_{\beta_2}\cdots s_{\beta_{n-2}}s_{\delta-\rho}$$
$$((s_{\beta_2}s_{\beta_1}s_{\beta_3}s_{\beta_2})(s_{\beta_4}s_{\beta_3}s_{\beta_5}s_{\beta_4})\cdots (s_{\beta_{n-4}}s_{\beta_{n-5}}s_{\beta_{n-3}}s_{\beta_{n-4}})s_{\delta-\rho})$$
$$\cdots((s_{\beta_2}s_{\beta_1}s_{\beta_3}s_{\beta_2})(s_{\beta_4}s_{\beta_3}s_{\beta_5}s_{\beta_4})s_{\delta-\rho})((s_{\beta_2}s_{\beta_1}s_{\beta_3}s_{\beta_2})s_{\delta-\rho})\underline{u})^{\infty}$$
where $\underline{u}\underline{v}$ is a reduced expression of $w_0^{W_{S\backslash \{s_{\beta_{n}}\}}}w_0$ where $n$ is even.

(4) For $W$ of type $E_6$, an infinite reduced word of $\widetilde{W}$ is minimal if and only if it has a reduced expression of the forms
$$(\underline{v}s_{\delta-\rho}(s_{\beta_2}s_{\beta_4}s_{\beta_5}s_{\beta_3}s_{\beta_4}s_{\beta_1}s_{\beta_2}s_{\beta_3}s_{\beta_4}s_{\beta_5} s_{\beta_6}s_{\delta-\rho})^2s_{\beta_2}s_{\beta_4}s_{\beta_5}s_{\beta_3}s_{\beta_4}s_{\beta_2}s_{\delta-\rho}\underline{u})^{\infty}$$
where $\underline{u}\underline{v}$ is a reduced expression of $w_0^{W_{S\backslash \{s_{\beta_{1}}\}}}w_0$,
$$(\underline{v}s_{\delta-\rho}\underline{u})^{\infty}$$
where $\underline{u}\underline{v}$ is a reduced expression of $w_0^{W_{S\backslash \{s_{\beta_{2}}\}}}w_0$,
$$(\underline{v}s_{\delta-\rho}(s_{\beta_2}s_{\beta_4}s_{\beta_5}s_{\beta_6}s_{\beta_3}s_{\beta_4}s_{\beta_5}s_{\beta_2}s_{\beta_4}s_{\beta_1}s_{\beta_3}s_{\beta_4}s_{\beta_2}s_{\beta_5}s_{\beta_4}s_{\beta_6}s_{\beta_5}s_{\delta-\rho})^2$$
$$(s_{\beta_2}s_{\beta_4}s_{\beta_3}s_{\beta_5}s_{\beta_4}s_{\beta_2}s_{\beta_1}s_{\beta_3}s_{\beta_4}s_{\beta_6}s_{\beta_5}s_{\beta_4}s_{\beta_3}s_{\beta_2}s_{\beta_4}s_{\delta-\rho})$$
$$(s_{\beta_2}s_{\beta_4}s_{\beta_3}s_{\beta_5}s_{\beta_6}s_{\beta_4}s_{\beta_1}s_{\beta_5}s_{\beta_3}s_{\beta_4}s_{\beta_2}s_{\delta-\rho})\underline{u})^{\infty}$$
where $\underline{u}\underline{v}$ is a reduced expression of $w_0^{W_{S\backslash \{s_{\beta_{3}}\}}}w_0$,
$$(\underline{v}s_{\delta-\rho}s_{\beta_2}s_{\beta_4}s_{\beta_3}s_{\beta_5}s_{\beta_6}s_{\beta_4}s_{\beta_1}s_{\beta_5}s_{\beta_3}s_{\beta_4}s_{\beta_2}s_{\delta-\rho}\underline{u})^{\infty}$$
where $\underline{u}\underline{v}$ is a reduced expression of $w_0^{W_{S\backslash \{s_{\beta_{4}}\}}}w_0$,
$$(\underline{v}s_{\delta-\rho}(s_{\beta_2}s_{\beta_4}s_{\beta_3}s_{\beta_1}s_{\beta_5}s_{\beta_4}s_{\beta_3}s_{\beta_2}s_{\beta_4}s_{\beta_6}s_{\beta_5}s_{\beta_4}s_{\beta_2}s_{\beta_3}s_{\beta_4}s_{\beta_1}s_{\beta_3}s_{\delta-\rho})^2$$
$$(s_{\beta_2}s_{\beta_4}s_{\beta_5}s_{\beta_3}s_{\beta_4}s_{\beta_2}s_{\beta_6}s_{\beta_5}s_{\beta_4}s_{\beta_1}s_{\beta_3}s_{\beta_4}s_{\beta_5}s_{\beta_2}s_{\beta_4}s_{\delta-\rho})$$
$$(s_{\beta_2}s_{\beta_4}s_{\beta_3}s_{\beta_5}s_{\beta_6}s_{\beta_4}s_{\beta_1}s_{\beta_5}s_{\beta_3}s_{\beta_4}s_{\beta_2}s_{\delta-\rho})\underline{u})^{\infty}$$
where $\underline{u}\underline{v}$ is a reduced expression of $w_0^{W_{S\backslash \{s_{\beta_{5}}\}}}w_0$,
$$(\underline{v}s_{\delta-\rho}(s_{\beta_2}s_{\beta_4}s_{\beta_3}s_{\beta_5}s_{\beta_4}s_{\beta_6}s_{\beta_2}s_{\beta_5}s_{\beta_4}s_{\beta_3} s_{\beta_1}s_{\delta-\rho})^2s_{\beta_2}s_{\beta_4}s_{\beta_3}s_{\beta_5}s_{\beta_4}s_{\beta_2}s_{\delta-\rho}\underline{u})^{\infty}$$
where $\underline{u}\underline{v}$ is a reduced expression of $w_0^{W_{S\backslash \{s_{\beta_{6}}\}}}w_0$.

(5) For $W$ of type $E_7$, an infinite reduced word of $\widetilde{W}$ is minimal if and only if it has a reduced expression of the forms
$$(\underline{v}s_{\delta-\rho}\underline{u})^{\infty}$$
where $\underline{u}\underline{v}$ is a reduced expression of $w_0^{W_{S\backslash \{s_{\beta_{1}}\}}}w_0$,
$$(\underline{v}s_{\delta-\rho}s_{\beta_1}s_{\beta_3}s_{\beta_4}s_{\beta_2}s_{\beta_5}s_{\beta_4}s_{\beta_3}s_{\beta_1}s_{\beta_6}s_{\beta_5}s_{\beta_4}s_{\beta_2}s_{\beta_3}s_{\beta_4}s_{\beta_5}$$
$$s_{\beta_7}s_{\beta_6}s_{\beta_5}s_{\beta_4}s_{\beta_3}s_{\beta_2}s_{\beta_1}s_{\beta_4}s_{\beta_3}s_{\beta_5}s_{\beta_4}s_{\beta_2}s_{\delta-\rho}s_{\beta_1}s_{\beta_3}s_{\beta_4}s_{\beta_2}s_{\beta_5}s_{\beta_4}s_{\beta_3}s_{\beta_1}s_{\beta_6}s_{\beta_5}$$
$$s_{\beta_4}s_{\beta_2}s_{\beta_3}s_{\beta_4}s_{\beta_7}s_{\beta_5}s_{\beta_6}s_{\delta-\rho}s_{\beta_1}s_{\beta_3}s_{\beta_4}s_{\beta_5}s_{\beta_2}s_{\beta_4}s_{\beta_3}s_{\beta_1}s_{\delta-\rho}\underline{u})^{\infty}$$
where $\underline{u}\underline{v}$ is a reduced expression of $w_0^{W_{S\backslash \{s_{\beta_{2}}\}}}w_0$,
$$(\underline{v}s_{\delta-\rho}s_{\beta_1}s_{\beta_3}s_{\beta_4}s_{\beta_5}s_{\beta_2}s_{\beta_6}s_{\beta_4}s_{\beta_7}s_{\beta_5}s_{\beta_3}s_{\beta_6}s_{\beta_4}s_{\beta_5}s_{\beta_2}s_{\beta_4}s_{\beta_3}s_{\beta_1}s_{\delta-\rho}\underline{u})^{\infty}$$
where $\underline{u}\underline{v}$ is a reduced expression of $w_0^{W_{S\backslash \{s_{\beta_{3}}\}}}w_0$,
$$(\underline{v}s_{\delta-\rho}s_{\beta_1}s_{\beta_3}s_{\beta_4}s_{\beta_2}s_{\beta_5}s_{\beta_4}s_{\beta_3}s_{\beta_1}s_{\beta_6}s_{\beta_5}s_{\beta_4}s_{\beta_3}s_{\beta_2}s_{\beta_4}s_{\beta_7}$$
$$s_{\beta_6}s_{\beta_5}s_{\beta_4}s_{\beta_2}s_{\beta_3}s_{\beta_4}s_{\beta_1}s_{\beta_3}s_{\delta-\rho}s_{\beta_1}s_{\beta_3}s_{\beta_4}s_{\beta_5}s_{\beta_2}$$
$$s_{\beta_6}s_{\beta_4}s_{\beta_7}s_{\beta_5}s_{\beta_3}s_{\beta_6}s_{\beta_4}s_{\beta_5}s_{\beta_2}s_{\beta_4}s_{\beta_3}s_{\beta_1}s_{\delta-\rho}\underline{u})^{\infty}$$
where $\underline{u}\underline{v}$ is a reduced expression of $w_0^{W_{S\backslash \{s_{\beta_{4}}\}}}w_0$,
$$(\underline{v}s_{\delta-\rho}s_{\beta_1}s_{\beta_3}s_{\beta_4}s_{\beta_5}s_{\beta_6}s_{\beta_7}s_{\beta_2}s_{\beta_4}s_{\beta_3}s_{\beta_5}s_{\beta_1}s_{\beta_4}s_{\beta_2}s_{\beta_6}$$
$$s_{\beta_3}s_{\beta_4}s_{\beta_5}s_{\beta_6}s_{\beta_7}s_{\beta_4}s_{\beta_3}s_{\beta_1}s_{\beta_2}s_{\beta_4}s_{\beta_5}s_{\beta_3}s_{\beta_6}s_{\beta_4}s_{\beta_5}s_{\delta-\rho}s_{\beta_1}s_{\beta_3}$$
$$s_{\beta_4}s_{\beta_5}s_{\beta_6}s_{\beta_7}s_{\beta_2}s_{\beta_4}s_{\beta_3}s_{\beta_5}s_{\beta_1}s_{\beta_4}s_{\beta_2}s_{\beta_6}s_{\beta_3}s_{\beta_4}s_{\beta_5}s_{\beta_6}s_{\beta_4}s_{\beta_3}s_{\beta_1}$$
$$s_{\beta_2}s_{\beta_4}s_{\beta_5}s_{\beta_3}s_{\beta_4}s_{\delta-\rho}s_{\beta_1}s_{\beta_3}s_{\beta_4}s_{\beta_2}s_{\beta_5}s_{\beta_4}s_{\beta_3}s_{\beta_1}s_{\beta_6}s_{\beta_5}s_{\beta_4}s_{\beta_3}s_{\beta_2}$$
$$s_{\beta_4}s_{\beta_7}s_{\beta_6}s_{\beta_5}s_{\beta_4}s_{\beta_2}s_{\beta_3}s_{\beta_4}s_{\beta_1}s_{\beta_3}s_{\delta-\rho}s_{\beta_1}s_{\beta_3}s_{\beta_4}s_{\beta_5}s_{\beta_2}s_{\beta_6}s_{\beta_4}s_{\beta_7}$$
$$s_{\beta_5}s_{\beta_3}s_{\beta_6}s_{\beta_4}s_{\beta_5}s_{\beta_2}s_{\beta_4}s_{\beta_3}s_{\beta_1}s_{\delta-\rho}\underline{u})^{\infty}$$
where $\underline{u}\underline{v}$ is a reduced expression of $w_0^{W_{S\backslash \{s_{\beta_{5}}\}}}w_0$,
$$(\underline{v}s_{\delta-\rho}s_{\beta_1}s_{\beta_3}s_{\beta_4}s_{\beta_5}s_{\beta_2}s_{\beta_4}s_{\beta_3}s_{\beta_1}s_{\delta-\rho}\underline{u})^{\infty}$$
where $\underline{u}\underline{v}$ is a reduced expression of $w_0^{W_{S\backslash \{s_{\beta_{6}}\}}}w_0$,
$$(\underline{v}s_{\delta-\rho}s_{\beta_1}s_{\beta_3}s_{\beta_4}s_{\beta_2}s_{\beta_5}s_{\beta_6}s_{\beta_4}s_{\beta_5}s_{\beta_3}s_{\beta_4}s_{\beta_1}s_{\beta_2}s_{\beta_3}s_{\beta_4}s_{\beta_5}s_{\beta_6}s_{\delta-\rho}$$
$$s_{\beta_1}s_{\beta_3}s_{\beta_4}s_{\beta_5}s_{\beta_2}s_{\beta_4}s_{\beta_3}s_{\beta_1}s_{\delta-\rho}\underline{u})^{\infty}$$
where $\underline{u}\underline{v}$ is a reduced expression of $w_0^{W_{S\backslash \{s_{\beta_{7}}\}}}w_0$.

(6) For $W$ of type $E_8$, an infinite reduced word of $\widetilde{W}$ is minimal if and only if it has a reduced expression of the forms
$$(\underline{v}s_{\delta-\rho}s_{\beta_8}s_{\beta_7}s_{\beta_6}s_{\beta_5}s_{\beta_4}s_{\beta_3}s_{\beta_2}s_{\beta_4}s_{\beta_5}s_{\beta_6}s_{\beta_7}s_{\beta_8}s_{\delta-\rho}\underline{u})^{\infty}$$
where $\underline{u}\underline{v}$ is a reduced expression of $w_0^{W_{S\backslash \{s_{\beta_{1}}\}}}w_0$,
$$(\underline{v}s_{\delta-\rho}s_{\beta_8}s_{\beta_7}s_{\beta_6}s_{\beta_5}s_{\beta_4}s_{\beta_2}s_{\beta_3}s_{\beta_4}s_{\beta_5}s_{\beta_6}s_{\beta_7}s_{\beta_8}s_{\beta_1}s_{\beta_3}s_{\beta_4}s_{\beta_5}s_{\beta_2}$$
$$s_{\beta_6}s_{\beta_4}s_{\beta_7}
s_{\beta_5}s_{\beta_3}s_{\beta_6}s_{\beta_4}s_{\beta_5}s_{\beta_2}s_{\beta_4}s_{\beta_3}s_{\beta_1}s_{\delta-\rho}s_{\beta_8}s_{\beta_7}s_{\beta_6}s_{\beta_5}s_{\beta_4}s_{\beta_3}s_{\beta_2}s_{\beta_4}s_{\beta_5}s_{\beta_6}$$
$$s_{\beta_7}s_{\beta_8}s_{\delta-\rho}\underline{u})^{\infty}$$
where $\underline{u}\underline{v}$ is a reduced expression of $w_0^{W_{S\backslash \{s_{\beta_{2}}\}}}w_0$,
$$(\underline{v}s_{\delta-\rho}s_{\beta_8}s_{\beta_7}s_{\beta_6}s_{\beta_5}s_{\beta_4}s_{\beta_3}s_{\beta_1}s_{\beta_2}s_{\beta_4}s_{\beta_3}s_{\beta_5}s_{\beta_4}s_{\beta_2}s_{\beta_6}s_{\beta_5}s_{\beta_4}s_{\beta_3}$$ $$s_{\beta_1}s_{\beta_7}s_{\beta_6}s_{\beta_5}s_{\beta_4}s_{\beta_2}s_{\beta_3}s_{\beta_4}s_{\beta_5}s_{\beta_8}s_{\beta_7}s_{\beta_6}s_{\beta_5}s_{\beta_4}s_{\beta_3}s_{\beta_2}s_{\beta_1}s_{\beta_4}s_{\beta_3}s_{\beta_5}s_{\beta_4}s_{\beta_2}s_{\delta-\rho}$$
$$s_{\beta_8}s_{\beta_7}s_{\beta_6}s_{\beta_5}s_{\beta_4}s_{\beta_2}s_{\beta_3}s_{\beta_4}s_{\beta_5}s_{\beta_6}s_{\beta_7}s_{\beta_8}s_{\beta_1}s_{\beta_3}s_{\beta_4}s_{\beta_5}s_{\beta_2}s_{\beta_6}s_{\beta_4}s_{\beta_7}s_{\beta_5}s_{\beta_3}s_{\beta_6}s_{\beta_4}$$
$$s_{\beta_5}s_{\beta_2}s_{\beta_4}s_{\beta_3}s_{\beta_1}s_{\delta-\rho}s_{\beta_8}s_{\beta_7}s_{\beta_6}s_{\beta_5}s_{\beta_4}s_{\beta_3}s_{\beta_2}s_{\beta_4}s_{\beta_5}s_{\beta_6}s_{\beta_7}s_{\beta_8}s_{\delta-\rho}\underline{u})^{\infty}$$
where $\underline{u}\underline{v}$ is a reduced expression of $w_0^{W_{S\backslash \{s_{\beta_{3}}\}}}w_0$,
$$(\underline{v}s_{\delta-\rho}s_{\beta_8}s_{\beta_7}s_{\beta_6}s_{\beta_5}s_{\beta_4}s_{\beta_3}s_{\beta_2}s_{\beta_4}s_{\beta_5}s_{\beta_6}s_{\beta_7}s_{\beta_8}s_{\beta_1}s_{\beta_3}s_{\beta_4}s_{\beta_2}s_{\beta_5}s_{\beta_4}s_{\beta_3}s_{\beta_1}s_{\beta_6}$$
$$s_{\beta_5}s_{\beta_4}s_{\beta_3}s_{\beta_2}s_{\beta_4}s_{\beta_5}s_{\beta_6}s_{\beta_7}s_{\beta_8}s_{\beta_6}s_{\beta_5}s_{\beta_4}s_{\beta_2}s_{\beta_3}s_{\beta_4}s_{\beta_5}s_{\beta_6}s_{\beta_7}s_{\beta_1}s_{\beta_3}s_{\beta_4}s_{\beta_5}s_{\beta_2}s_{\beta_6}s_{\beta_4}$$
$$s_{\beta_5}s_{\delta-\rho}s_{\beta_8}s_{\beta_7}s_{\beta_6}s_{\beta_5}s_{\beta_4}s_{\beta_2}s_{\beta_3}s_{\beta_4}s_{\beta_1}s_{\beta_3}s_{\beta_5}s_{\beta_6}s_{\beta_7}s_{\beta_8}s_{\beta_4}s_{\beta_5}s_{\beta_6}s_{\beta_7}s_{\beta_2}s_{\beta_4}s_{\beta_3}s_{\beta_5}s_{\beta_1}$$
$$s_{\beta_4}s_{\beta_2}s_{\beta_6}s_{\beta_3}s_{\beta_4}s_{\beta_5}s_{\beta_6}s_{\beta_7}s_{\beta_8}s_{\beta_4}s_{\beta_3}s_{\beta_1}s_{\beta_2}s_{\beta_4}s_{\beta_5}s_{\beta_3}s_{\beta_6}s_{\beta_4}s_{\beta_7}s_{\beta_5}s_{\beta_6}s_{\delta-\rho}s_{\beta_8}s_{\beta_7}s_{\beta_6}$$
$$s_{\beta_5}s_{\beta_4}s_{\beta_3}s_{\beta_2}s_{\beta_1}s_{\beta_4}s_{\beta_3}s_{\beta_5}s_{\beta_4}s_{\beta_6}s_{\beta_7}s_{\beta_8}s_{\beta_5}s_{\beta_6}s_{\beta_7}s_{\beta_2}s_{\beta_4}s_{\beta_5}s_{\beta_6}s_{\beta_3}s_{\beta_4}s_{\beta_2}s_{\beta_5}s_{\beta_4}s_{\beta_1}s_{\beta_3}$$
$$s_{\beta_4}s_{\beta_5}s_{\beta_2}s_{\beta_6}s_{\beta_4}s_{\beta_7}s_{\beta_5}s_{\beta_8}s_{\beta_6}s_{\beta_7}s_{\delta-\rho}s_{\beta_8}s_{\beta_7}s_{\beta_6}s_{\beta_5}s_{\beta_4}s_{\beta_2}s_{\beta_3}s_{\beta_4}s_{\beta_1}s_{\beta_5}s_{\beta_3}s_{\beta_6}s_{\beta_4}s_{\beta_7}$$
$$s_{\beta_5}s_{\beta_2}s_{\beta_6}s_{\beta_4}s_{\beta_5}s_{\beta_3}s_{\beta_4}s_{\beta_1}s_{\beta_3}s_{\beta_2}s_{\beta_4}s_{\beta_5}s_{\beta_6}s_{\beta_7}s_{\beta_8}s_{\delta-\rho}\underline{u})^{\infty}$$
where $\underline{u}\underline{v}$ is a reduced expression of $w_0^{W_{S\backslash \{s_{\beta_{4}}\}}}w_0$,
$$(\underline{v}s_{\delta-\rho}s_{\beta_8}s_{\beta_7}s_{\beta_6}s_{\beta_5}s_{\beta_4}s_{\beta_2}s_{\beta_3}s_{\beta_4}s_{\beta_1}s_{\beta_3}s_{\beta_5}s_{\beta_6}s_{\beta_7}s_{\beta_8}s_{\beta_4}s_{\beta_5}s_{\beta_6}s_{\beta_7}s_{\beta_2}s_{\beta_4}s_{\beta_3}s_{\beta_5}s_{\beta_1}$$
$$s_{\beta_4}s_{\beta_2}s_{\beta_6}s_{\beta_3}s_{\beta_4}s_{\beta_5}s_{\beta_6}s_{\beta_7}s_{\beta_8}s_{\beta_4}s_{\beta_3}s_{\beta_1}s_{\beta_2}s_{\beta_4}s_{\beta_5}s_{\beta_3}s_{\beta_6}s_{\beta_4}s_{\beta_7}s_{\beta_5}s_{\beta_6}s_{\delta-\rho}s_{\beta_8}s_{\beta_7}s_{\beta_6}s_{\beta_5}$$
$$s_{\beta_4}s_{\beta_3}s_{\beta_2}s_{\beta_1}s_{\beta_4}s_{\beta_3}s_{\beta_5}s_{\beta_4}s_{\beta_6}s_{\beta_7}s_{\beta_8}s_{\beta_5}s_{\beta_6}s_{\beta_7}s_{\beta_2}s_{\beta_4}s_{\beta_5}s_{\beta_6}s_{\beta_3}s_{\beta_4}s_{\beta_2}s_{\beta_5}s_{\beta_4}s_{\beta_1}s_{\beta_3}s_{\beta_4}$$
$$s_{\beta_5}s_{\beta_2}s_{\beta_6}s_{\beta_4}s_{\beta_7}s_{\beta_5}s_{\beta_8}s_{\beta_6}s_{\beta_7}s_{\delta-\rho}s_{\beta_8}s_{\beta_7}s_{\beta_6}s_{\beta_5}s_{\beta_4}s_{\beta_2}s_{\beta_3}s_{\beta_4}s_{\beta_1}s_{\beta_5}s_{\beta_3}s_{\beta_6}s_{\beta_4}s_{\beta_7}s_{\beta_5}$$
$$s_{\beta_2}s_{\beta_6}s_{\beta_4}s_{\beta_5}s_{\beta_3}s_{\beta_4}s_{\beta_1}s_{\beta_3}s_{\beta_2}s_{\beta_4}s_{\beta_5}s_{\beta_6}s_{\beta_7}s_{\beta_8}s_{\delta-\rho}\underline{u})^{\infty}$$
where $\underline{u}\underline{v}$ is a reduced expression of $w_0^{W_{S\backslash \{s_{\beta_{5}}\}}}w_0$,
$$(\underline{v}s_{\delta-\rho}s_{\beta_8}s_{\beta_7}s_{\beta_6}s_{\beta_5}s_{\beta_4}s_{\beta_3}s_{\beta_2}s_{\beta_1}s_{\beta_4}s_{\beta_3}s_{\beta_5}s_{\beta_4}s_{\beta_6}s_{\beta_7}s_{\beta_8}s_{\beta_5}s_{\beta_6}s_{\beta_7}s_{\beta_2}s_{\beta_4}s_{\beta_5}$$
$$s_{\beta_6}s_{\beta_3}s_{\beta_4}s_{\beta_2}s_{\beta_5}s_{\beta_4}s_{\beta_1}s_{\beta_3}s_{\beta_4}s_{\beta_5}s_{\beta_2}s_{\beta_6}s_{\beta_4}s_{\beta_7}s_{\beta_5}s_{\beta_8}s_{\beta_6}s_{\beta_7}s_{\delta-\rho}s_{\beta_8}s_{\beta_7}s_{\beta_6}s_{\beta_5}s_{\beta_4}$$
$$s_{\beta_2}s_{\beta_3}s_{\beta_4}s_{\beta_1}s_{\beta_5}s_{\beta_3}s_{\beta_6}s_{\beta_4}s_{\beta_7}s_{\beta_5}s_{\beta_2}s_{\beta_6}s_{\beta_4}s_{\beta_5}s_{\beta_3}s_{\beta_4}s_{\beta_1}s_{\beta_3}s_{\beta_2}s_{\beta_4}s_{\beta_5}s_{\beta_6}s_{\beta_7}s_{\beta_8}$$
$$s_{\delta-\rho}\underline{u})^{\infty}$$
where $\underline{u}\underline{v}$ is a reduced expression of $w_0^{W_{S\backslash \{s_{\beta_{6}}\}}}w_0$,
$$(\underline{v}s_{\delta-\rho}s_{\beta_8}s_{\beta_7}s_{\beta_6}s_{\beta_5}s_{\beta_4}s_{\beta_2}s_{\beta_3}s_{\beta_4}s_{\beta_1}s_{\beta_5}s_{\beta_3}s_{\beta_6}s_{\beta_4}s_{\beta_7}s_{\beta_5}s_{\beta_2}s_{\beta_6}s_{\beta_4}s_{\beta_5}s_{\beta_3}s_{\beta_4}$$
$$s_{\beta_1}s_{\beta_3}s_{\beta_2}s_{\beta_4}s_{\beta_5}s_{\beta_6}s_{\beta_7}s_{\beta_8}s_{\delta-\rho}\underline{u})^{\infty}$$
where $\underline{u}\underline{v}$ is a reduced expression of $w_0^{W_{S\backslash \{s_{\beta_{7}}\}}}w_0$,
$$(\underline{v}s_{\delta-\rho}\underline{u})^{\infty}$$
where $\underline{u}\underline{v}$ is a reduced expression of $w_0^{W_{S\backslash \{s_{\beta_{8}}\}}}w_0$.

(7) For $W$ of type $F_4$, an infinite reduced word of $\widetilde{W}$ is minimal if and only if it has a reduced expression of the forms
$$(\underline{v}s_{\delta-\rho}\underline{u})^{\infty}$$
where $\underline{u}\underline{v}$ is a reduced expression of $w_0^{W_{S\backslash \{s_{\beta_{1}}\}}}w_0$,
$$(\underline{v}s_{\delta-\rho}s_{\beta_1}s_{\beta_2}s_{\beta_3}s_{\beta_4}s_{\beta_2}s_{\beta_3}s_{\beta_2}s_{\beta_1}s_{\delta-\rho}\underline{u})^{\infty}$$
where $\underline{u}\underline{v}$ is a reduced expression of $w_0^{W_{S\backslash \{s_{\beta_{2}}\}}}w_0$,
$$(\underline{v}s_{\delta-\rho}s_{\beta_1}s_{\beta_2}s_{\beta_3}s_{\beta_2}s_{\beta_1}s_{\beta_4}s_{\beta_3}s_{\beta_2}s_{\beta_3}s_{\beta_1}s_{\beta_2}s_{\delta-\rho}$$
$$s_{\beta_1}s_{\beta_2}s_{\beta_3}s_{\beta_4}s_{\beta_2}s_{\beta_3}s_{\beta_2}s_{\beta_1}s_{\delta-\rho}\underline{u})^{\infty}$$
where $\underline{u}\underline{v}$ is a reduced expression of $w_0^{W_{S\backslash \{s_{\beta_{3}}\}}}w_0$,
$$(\underline{v}s_{\delta-\rho}s_{\beta_1}s_{\beta_2}s_{\beta_3}s_{\beta_2}s_{\beta_1}s_{\delta-\rho}\underline{u})^{\infty}$$
where $\underline{u}\underline{v}$ is a reduced expression of $w_0^{W_{S\backslash \{s_{\beta_{4}}\}}}w_0$.

(8) For $W$ of type $G_2$, an infinite reduced word of $\widetilde{W}$ is minimal if and only if it has a reduced expression of the forms
$$(\underline{v}s_{\delta-\rho}s_{\beta_2}s_{\beta_1}s_{\beta_2}s_{\delta-\rho}\underline{u})^{\infty}$$
where $\underline{u}\underline{v}$ is a reduced expression of $w_0^{W_{S\backslash \{s_{\beta_{1}}\}}}w_0$,
$$(\underline{v}s_{\delta-\rho}\underline{u})^{\infty}$$
where $\underline{u}\underline{v}$ is a reduced expression of $w_0^{W_{S\backslash \{s_{\beta_{2}}\}}}w_0$.
\end{theorem}

To prove the theorem, we begin with the following lemmas.


\begin{lemma}\label{first}
 For $w\in W$, one has  $w\Phi^+_{\Delta_1,\emptyset}=u\Phi^+_{\Delta_1,\emptyset}$ such that $u^{-1}$ is a left prefix of $w_0^{W_{\Delta_1}}w_0$.

 Furthermore, suppose that

 (1) $(w_0^{W_{\Delta_1}}w_0)x$ is straight (in $\widetilde{W}$) and $w_0^{W_{\Delta_1}}w_0$ is a left prefix of $(w_0^{W_{\Delta_1}}w_0)x$,

 (2) $\Phi_{((w_0^{W_{\Delta_1}}w_0)x)^{\infty}}=\widehat{\Phi^+_{\Delta_1,\emptyset}}$.

 Then $\widehat{w\Phi^+_{\Delta_1,\emptyset}}=\Phi_{(u(w_0^{W_{\Delta_1}}w_0)xu^{-1})^{\infty}}$.
\end{lemma}

\begin{proof}
It is well known that there exists a decomposition $w=uv$ such that $\ell(w)=\ell(u)+\ell(v)$, $v\in W_{\Delta_1}$ and no reduced expression of $u$ ends with $s_{\alpha},\alpha\in \Delta_1.$ (See Proposition 2.2.4 in \cite{bjornerbrenti}) Since $v$ stabilizes $\Phi^+_{\Delta_1,\emptyset}$, $w\Phi^+_{\Delta_1,\emptyset}=u\Phi^+_{\Delta_1,\emptyset}$. By Theorem 4.1 of \cite{bjornerquotient}, under the weak left order there exists a unique maximal element $p$ of the set $\{q\in W|\,l(qs_{\alpha})>l(q), \forall \alpha\in \Delta_1\}$ and therefore $u\leq p$ under the weak left order. And such an element is exactly $w_0w^{W_{\Delta_1}}_0$ (see section 2.5 of \cite{bjornerbrenti}). So the first conclusion follows.

Now we have
$$\widehat{w\Phi^+_{\Delta_1,\emptyset}}=\widehat{u\Phi^+_{\Delta_1,\emptyset}}=u\cdot \widehat{\Phi^+_{\Delta_1,\emptyset}}=u\cdot\Phi_{((w_0^{W_{\Delta_1}}w_0)x)^{\infty}}$$
$$=\Phi_{u((w_0^{W_{\Delta_1}}w_0)x)^{\infty}}=\Phi_{(u(w_0^{W_{\Delta_1}}w_0)xu^{-1})^{\infty}}.$$
The second equality follows from Proposition 5.13(b) in \cite{DyerReflOrder} and the fact that $u$ is the element of  minimal length in the coset $wW_{\Delta_1}$. The fourth equality follows from Lemma 2.7(d) in \cite{wang}.
\end{proof}

Given $w\in W$ and $\Phi^+_{\Delta_1,\emptyset}$, denote the unique element $u$ as in the above Lemma by $u[w,\Delta_1]$.

We call a coroot $\beta$ an $i$-th distinguished coroot if $(\beta_i,\beta)>0$ and $(\beta_j,\beta)=0$ for $j\neq i.$
We call  $t_{\beta}$ an  $i$-th distinguished ``translation'' if $\beta$ is an $i$-th distinguished coroot.

\begin{lemma}\label{two}
(1) Let $\beta$ be an $i$-th distinguished coroot. One has $\Phi^+_{\Delta\backslash\{\beta_i\},\emptyset} \subset \Phi_{t_{\beta}}$ and $\widehat{\Phi^+_{\Delta\backslash\{\beta_i\},\emptyset}}=\Phi_{t_{\beta}^{\infty}}$. Therefore $t_{\beta}^{\infty}$ is the minimal infinite reduced word whose inversion set is $\widehat{\Phi^+_{\Delta\backslash\{\beta_i\},\emptyset}}$.

(2) The element $w_0^{W_{S\backslash \{s_{\beta_i}\}}}w_0$ is a left prefix of $t_{\beta}$. Then $t_{\beta}=w_0^{W_{S\backslash \{s_{\beta_i}\}}}w_0x$ for some $x\in W$. For any $w\in W$,
$\widehat{w\Phi^+_{\Delta\backslash\{\beta_i\},\emptyset}}=\Phi_{(u[w,\Delta_1](w_0^{W_{\Delta_1}}w_0)x(u[w,\Delta_1])^{-1})^{\infty}}$ where $\Delta_1=\Delta\backslash\{\beta_i\}$.
\end{lemma}

\begin{proof}
(1) follows from Lemma 4.6 and its proof in \cite{orderpaper}.

(2) The first assertion follows from the fact $$\Phi_{w_0^{W_{S\backslash \{s_{\beta_i}\}}}w_0}=\Phi^+\backslash \Phi^+_{S\backslash \{s_{\beta_i}\}}=\Phi^+_{\Delta\backslash\{\beta_i\},\emptyset}$$ and $\Phi^+_{\Delta\backslash\{\beta_i\},\emptyset} \subset \Phi_{t_{\beta}}$ (by (1)).
 The other assertions follow directly from Lemma
\ref{first}.
\end{proof}

\emph{Proof of Theorem \ref{main}}. A minimal infinite reduced word of $\widetilde{W}$ has its inversion set of the form $\widehat{\Psi^+_{\Pi_1,\emptyset}}$ where $|\Pi_1|$ equals $\mathrm{rank}(W)-1$ by Theorem \ref{thmaffineminimal}. We have that $\widehat{\Psi^+_{\Pi_1,\emptyset}}=\widehat{w\Phi^+_{\Delta_1,\emptyset}}$ for some $w\in W$ and $\Delta_1\subset \Delta, |\Delta_1|=\mathrm{rank}(W)-1.$ Thanks to Lemma \ref{two} (2), we only need to verify that for each of the expressions of the form $(\underline{v}\underline{y}\underline{u})^{\infty}$ in the Theorem, $\underline{uvy}$ is a reduced expression for some distinguished ``translation''. We leave the detail verification in the following sections. $\Box$

\section{Reduced expression of distinguished ``translations''}\label{translation}

From the previous section, we have shown that the proof  of Theorem \ref{main} boils down to finding a reduced expression of some distinguished ``translations". In this section we will describe these reduced expressions explicitly. Due to their length, we will only present the full detail for type $\widetilde{C_n}$ and omit the (virtually same but long) verification arguments for other types.

For the realization of the root system,  we roughly follow 18.14 of \cite{ty} with rescaling in some cases. When writing a positive root $\beta$  as $\sum k_i\beta_i$ where $\beta_i$ is simple, we define the support of $\beta$ to be the set of simple root $\beta_j$ such that $k_j>0$.

\subsection{$C_n$}

The simple roots are $\beta_1,\beta_2,\cdots, \beta_n$ and they are numbered as in Section \ref{SectMain}. Let $\epsilon_1, \epsilon_2,\cdots, \epsilon_n$ be the standard basis of $\mathbb{R}^n$. Then simple roots can be realized as follow: $\beta_1=\frac{\epsilon_1-\epsilon_2}{\sqrt{2}}, \beta_2=\frac{\epsilon_2-\epsilon_3}{\sqrt{2}}, \cdots \beta_{n-1}=\frac{\epsilon_{n-1}-\epsilon_n}{\sqrt{2}}, \beta_n=\sqrt{2}\epsilon_n.$ Therefore the coroots are $2\beta_1,2\beta_2,\cdots, 2\beta_{n-1}, \beta_n.$ The positive roots are:
$2(\beta_i+\beta_{i+1}+\cdots+\beta_{n-1})+\beta_n, 1\leq i\leq n, \beta_i+\beta_{i+1}+\cdots+\beta_{j-1}, 1\leq i<j\leq n, \beta_i+\beta_{i+1}+\cdots+\beta_{j-1}+2(\beta_j+\beta_{j+1}+\cdots+\beta_{n-1})+\beta_n, 1\leq i<j\leq n$. The highest root is $2(\beta_1+\beta_2+\cdots+\beta_{n-1})+\beta_n$.

\begin{lemma}
(1) A set of $i$-th distinguished ``translations'' $1\leq i\leq n$ is given as follow.$$t_{2\beta_1+4\beta_2+\cdots+2(i-1)\beta_{i-1}+2i\beta_i+2i\beta_{i+1}+\cdots+2i\beta_{n-1}+i\beta_n}$$
$$=w_0^{W_{S\backslash \{s_{\beta_i}\}}}w_0s_{\delta-\rho}(s_{\beta_1}s_{\delta-\rho})(s_{\beta_2}s_{\beta_1}s_{\delta-\rho})\cdots(s_{\beta_{i-1}}s_{\beta_{i-2}}\cdots s_{\beta_1}s_{\delta-\rho})$$

(2) Replace $w_0^{W_{S\backslash \{s_{\beta_i}\}}}w_0$ with a reduced expression of it, the expressions in (1) give  reduced expressions of these ``translations''.
\end{lemma}

\begin{proof}
(1) One easily checks that the ``translations'' in the lemma are distinguished.

It is well-known from the action of the longest element on a simple root that
$$w_0^{W_{S\backslash \{s_{\beta_i}\}}}w_0(\beta_1)=\beta_{i-1},$$
$$w_0^{W_{S\backslash \{s_{\beta_i}\}}}w_0(\beta_2)=\beta_{i-2},$$
$$\cdots$$
$$w_0^{W_{S\backslash \{s_{\beta_i}\}}}w_0(\beta_{i-1})=\beta_{1},$$
$$w_0^{W_{S\backslash \{s_{\beta_i}\}}}w_0(\beta_{i})=-\beta_1-\beta_2-\cdots-\beta_i-2\beta_{i+1}-2\beta_{i+2}-\cdots-2\beta_{n-1}-\beta_n,\text{if}\,i\neq n,$$$$w_0^{W_{S\backslash \{s_{\beta_i}\}}}w_0(\beta_{i})=-2\beta_1-2\beta_2-\cdots-2\beta_{n-1}-\beta_n,\text{if}\,i=n,$$
$$w_0^{W_{S\backslash \{s_{\beta_i}\}}}w_0(\beta_{i+1})=\beta_{i+1},$$
$$\cdots$$
$$w_0^{W_{S\backslash \{s_{\beta_i}\}}}w_0(\beta_{n})=\beta_{n}.$$
Now we check the identity by acting the left hand side and the right hand side on simple roots. For $j<i$,
$$s_{\beta_{i-1}}s_{\beta_{i-2}}\cdots s_{\beta_1}s_{\delta-\rho}(\beta_j)=\beta_{j-1},$$
$$s_{\beta_{i-2}}s_{\beta_{i-3}}\cdots s_{\beta_1}s_{\delta-\rho}(\beta_{j-1})=\beta_{j-2},$$
$$\cdots$$
$$s_{\beta_{i-(j-1)}}s_{\beta_{i-j}}\cdots s_{\beta_1}s_{\delta-\rho}(\beta_{2})=\beta_{1},$$
$$s_{\beta_{i-j}}s_{\beta_{i-j-1}}\cdots s_{\beta_1}s_{\delta-\rho}(\beta_1)=\delta-\beta_1-\cdots-\beta_{i-j}-2\beta_{i-j+1}-\cdots-2\beta_{n-1}-\beta_n,$$
$$s_{\beta_{i-j-1}}\cdots s_{\beta_1}s_{\delta-\rho}(\delta-\beta_1-\cdots-\beta_{i-j}-2\beta_{i-j+1}-\cdots-2\beta_{n-1}-\beta_n)=\beta_{i-j},$$
$$s_{\beta_{i-j-2}}\cdots s_{\beta_1}s_{\delta-\rho}(\beta_{i-j})=\beta_{i-j},$$
$$\cdots$$
$$s_{\delta-\rho}(\beta_{i-j})=\beta_{i-j}.$$
$$w_0^{W_{S\backslash \{s_{\beta_n}\}}}w_0(\beta_{i-j})=\beta_j.$$
For $j>i$ one easily sees that $$w_0^{W_{S\backslash \{s_{\beta_i}\}}}w_0s_{\delta-\rho}(s_{\beta_1}s_{\delta-\rho})(s_{\beta_2}s_{\beta_1}s_{\delta-\rho})\cdots(s_{\beta_{i-1}}s_{\beta_{i-2}}\cdots s_{\beta_1}s_{\delta-\rho})(\beta_j)=\beta_j.$$

Suppose that $i\neq n$ we compute
$$s_{\delta-\rho}(s_{\beta_1}s_{\delta-\rho})(s_{\beta_2}s_{\beta_1}s_{\delta-\rho})\cdots(s_{\beta_{i-1}}s_{\beta_{i-2}}\cdots s_{\beta_1}s_{\delta-\rho})(\beta_i)$$
$$=\delta-\beta_1-\cdots-\beta_i-2\beta_{i+1}-\cdots-2\beta_{n-1}-\beta_n,$$
$$w_0^{W_{S\backslash \{s_{\beta_i}\}}}w_0(\delta-\beta_1-\cdots-\beta_i-2\beta_{i+1}-\cdots-2\beta_{n-1}-\beta_n)=\beta_i+\delta.$$
Now assume that $i=n$. We compute
$$s_{\delta-\rho}(s_{\beta_1}s_{\delta-\rho})(s_{\beta_2}s_{\beta_1}s_{\delta-\rho})\cdots(s_{\beta_{n-1}}s_{\beta_{i-2}}\cdots s_{\beta_1}s_{\delta-\rho})(\beta_n)$$
$$=2\delta-2\beta_1-\cdots-2\beta_{n-1}-\beta_n,$$
$$w_0^{W_{S\backslash \{s_{\beta_n}\}}}w_0(2\delta-2\beta_1-\cdots-2\beta_{n-1}-\beta_n)=\beta_n+2\delta.$$

On the other hand, $$t_{2\beta_1+4\beta_2+\cdots+2(i-1)\beta_{i-1}+2i\beta_i+2i\beta_{i+1}+\cdots+2i\beta_{n-1}+i\beta_n}(\beta_j)=\beta_j, j\neq i,$$ $$t_{2\beta_1+4\beta_2+\cdots+2(i-1)\beta_{i-1}+2i\beta_i+2i\beta_{i+1}+\cdots+2i\beta_{n-1}+i\beta_n}(\beta_i)=\beta_i+\delta, \text{if}\, i\neq n$$
$$t_{2\beta_1+4\beta_2+\cdots+2(n-1)\beta_{n-1}+n\beta_n}(\beta_n)=\beta_n+2\delta.$$
Therefore combining the calculation and the faithfulness of the reflection representation, we have proved (1).

(2) Let $i<n$. There are $n^2-\frac{i(i-1)}{2}-(n-i)^2$ positive roots whose support containing $\beta_i$. Among them $i+\frac{i(i-1)}{2}$ of them has the coefficient of $\beta_i$ being 2 ($i$ of them are of the form $2(\beta_k+\cdots+\beta_n)+\beta_n$ and $\frac{i(i-1)}{2}$ of them are of the form $\beta_k+\cdots+\beta_{j-1}+2(\beta_j+\cdots+\beta_{n-1})+\beta_n$) and others have the coefficient of $\beta_i$ being 1.  Hence $$|\Phi_{t_{2\beta_1+4\beta_2+\cdots+2(i-1)\beta_{i-1}+2i\beta_i+2i\beta_{i+1}+\cdots+2i\beta_{n-1}+i\beta_n}}|$$
$$=(n^2-\frac{i(i-1)}{2}-(n-i)^2)-(i+\frac{i(i-1)}{2})+2(i+\frac{i(i-1)}{2})=-i^2+i+2ni.$$

Let $i=n$. There are $n^2-\frac{n(n-1)}{2}$ positive roots whose support containing $\beta_i$ ($\beta_n$). All of them have the coefficient of $\beta_n$ being 1. Hence $$|\Phi_{t_{2\beta_1+4\beta_2+\cdots+2(n-1)\beta_{n-1}+n\beta_n}}|$$
$$=2(n^2-\frac{n(n-1)}{2})=n^2+n=-i^2+i+2ni.$$

Then
$$\ell(w_0^{W_{S\backslash \{s_{\beta_i}\}}}w_0s_{\delta-\rho}(s_{\beta_1}s_{\delta-\rho})(s_{\beta_2}s_{\beta_1}s_{\delta-\rho})\cdots(s_{\beta_{i-1}}s_{\beta_{i-2}}\cdots s_{\beta_1}s_{\delta-\rho}))$$
$$\leq \ell(w_0^{W_{S\backslash \{s_{\beta_i}\}}}w_0)+\ell(s_{\delta-\rho}(s_{\beta_1}s_{\delta-\rho})(s_{\beta_2}s_{\beta_1}s_{\delta-\rho})\cdots(s_{\beta_{i-1}}s_{\beta_{i-2}}\cdots s_{\beta_1}s_{\delta-\rho}))$$
$$=(n^2-\frac{i(i-1)}{2}-(n-i)^2)+(1+2+\cdots+i)=-i^2+i+2ni$$
$$=\ell(t_{2\beta_1+4\beta_2+\cdots+2(i-1)\beta_{i-1}+2i\beta_i+2i\beta_{i+1}+\cdots+2i\beta_{n-1}+i\beta_n})$$
But by (1) the inequality is indeed equality.
\end{proof}

\subsection{$B_n$}

The simple roots are $\beta_1,\beta_2,\cdots, \beta_n$ and they are numbered as in Section \ref{SectMain}. Let $\epsilon_1, \epsilon_2,\cdots, \epsilon_n$ be the standard basis of $\mathbb{R}^n$. Then simple roots can be realized as follow: $\beta_1=\epsilon_1-\epsilon_2, \beta_2=\epsilon_2-\epsilon_3, \cdots \beta_{n-1}=\epsilon_{n-1}-\epsilon_n, \beta_n=\epsilon_n.$ Therefore the coroots are $\beta_1,\beta_2,\cdots, \beta_{n-1}, 2\beta_n.$ The positive roots are:
$\beta_i+\beta_{i+1}+\cdots+\beta_n, 1\leq i\leq n, \beta_i+\beta_{i+1}+\cdots+\beta_{j-1}, 1\leq i<j\leq n, \beta_i+\beta_{i+1}+\cdots+\beta_{j-1}+2(\beta_j+\beta_{j+1}+\cdots+\beta_n), 1\leq i<j\leq n$. The highest root is $\beta_1+2(\beta_2+\beta_3+\cdots+\beta_n)$.

\begin{lemma}
(1) A set of $i$-th distinguished ``translations'' $1\leq i\leq n$ is given as follow.

When $i$ is even,
$$t_{\beta_1+2\beta_2+\cdots+i\beta_i+i\beta_{i+1}+\cdots i\beta_n}$$
$$=w_0^{W_{S\backslash \{s_{\beta_i}\}}}w_0s_{\delta-\rho}((s_{\beta_2}s_{\beta_1}s_{\beta_3}s_{\beta_2})(s_{\beta_4}s_{\beta_3}s_{\beta_5}s_{\beta_4})\cdots (s_{\beta_{i-2}}s_{\beta_{i-3}}s_{\beta_{i-1}}s_{\beta_{i-2}})s_{\delta-\rho})$$
$$\cdots((s_{\beta_2}s_{\beta_1}s_{\beta_3}s_{\beta_2})(s_{\beta_4}s_{\beta_3}s_{\beta_5}s_{\beta_4})s_{\delta-\rho}))((s_{\beta_2}s_{\beta_1}s_{\beta_3}s_{\beta_2})s_{\delta-\rho}))$$

When $i$ is odd and $i\geq 3$
$$t_{2\beta_1+4\beta_2+\cdots+2i\beta_i+2i\beta_{i+1}+\cdots 2i\beta_n}$$
$$=w_0^{W_{S\backslash \{s_{\beta_i}\}}}w_0(s_{\delta-\rho}(s_{\beta_2}s_{\beta_3}\cdots s_{\beta_n}s_{\beta_{n-1}}\cdots s_{\beta_i}s_{\beta_1}s_{\beta_2}\cdots s_{\beta_n}s_{\beta_{n-1}}\cdots s_{\beta_i}))^{\frac{i-1}{2}}$$
$$(s_{\delta-\rho}s_{\beta_2}s_{\beta_3}\cdots s_{\beta_n}s_{\beta_{n-1}}\cdots s_{\beta_i}s_{\beta_1}s_{\beta_2}\cdots s_{\beta_{i-1}})$$
$$(s_{\delta-\rho}(s_{\beta_2}s_{\beta_1}s_{\beta_3}s_{\beta_2})(s_{\beta_4}s_{\beta_3}s_{\beta_5}s_{\beta_4})\cdots (s_{\beta_{i-3}}s_{\beta_{i-4}}s_{\beta_{i-2}}s_{\beta_{i-3}}))$$
$$\cdots (s_{\delta-\rho}(s_{\beta_2}s_{\beta_1}s_{\beta_3}s_{\beta_2})(s_{\beta_4}s_{\beta_3}s_{\beta_5}s_{\beta_4}))(s_{\delta-\rho}(s_{\beta_2}s_{\beta_1}s_{\beta_3}s_{\beta_2}))s_{\delta-\rho}.$$

$$t_{2\beta_1+2\beta_2+\cdots+2\beta_n}$$
$$=w_0^{W_{S\backslash \{s_{\beta_1}\}}}w_0s_{\delta-\rho}s_{\beta_2}s_{\beta_3}\cdots s_{\beta_n}\cdots s_{\beta_3}s_{\beta_2}s_{\delta-\rho}$$

(2) Replace $w_0^{W_{S\backslash \{s_{\beta_i}\}}}w_0$ with a reduced expression of it, the expressions in (1) give  reduced expressions of these ``translations''.
\end{lemma}

\subsection{$D_n$}

The simple roots are $\beta_1,\beta_2,\cdots, \beta_n$ and they are numbered as in Section \ref{SectMain}. Let $\epsilon_1, \epsilon_2,\cdots, \epsilon_n$ be the standard basis of $\mathbb{R}^n$. Then simple roots can be realized as follow: $\beta_1=\frac{\epsilon_1-\epsilon_2}{\sqrt{2}}, \beta_2=\frac{\epsilon_2-\epsilon_3}{\sqrt{2}}, \cdots \beta_{n-1}=\frac{\epsilon_{n-1}-\epsilon_n}{\sqrt{2}}, \beta_n=\frac{\epsilon_{n-1}+\epsilon_n}{\sqrt{2}}.$ Therefore the coroots are $2\beta_1,2\beta_2,\cdots, 2\beta_{n-1}, 2\beta_n.$ The positive roots are:
$\beta_i+\beta_{i+1}+\cdots+\beta_{j-1}, 1\leq i<j\leq n, \beta_i+\beta_{i+1}+\cdots+\beta_{n-2}+\beta_n, 1\leq i<n, \beta_i+\beta_{i+1}+\cdots+\beta_{j-1}+2(\beta_j+\beta_{j+1}+\cdots+\beta_{n-2})+\beta_{n-1}+\beta_n, 1\leq i<j<n$. The highest root is $\beta_1+2(\beta_2+\beta_3+\cdots+\beta_{n-2})+\beta_{n-1}+\beta_n$.

\begin{lemma}
(1) A set of $i$-th distinguished ``translations'' $1\leq i\leq n$ is given as follow.
$$t_{4\beta_1+4\beta_2+\cdots+4\beta_{n-2}+2\beta_{n-1}+2\beta_n}$$
$$=w_0^{W_{S\backslash \{s_{\beta_1}\}}}w_0s_{\delta-\rho}s_{\beta_2}s_{\beta_3}\cdots s_{\beta_{n-2}}s_{\beta_{n-1}}s_{\beta_{n}}s_{\beta_{n-2}}\cdots s_{\beta_3}s_{\beta_2}s_{\delta-\rho}$$

Let $i$ be even and $i\leq n-2$.
$$t_{2\beta_1+4\beta_2+\cdots+2i\beta_i+2i\beta_{i+1}+\cdots+2i\beta_{n-2}+i\beta_{n-1}+i\beta_n}$$
$$=w_0^{W_{S\backslash \{s_{\beta_i}\}}}w_0s_{\delta-\rho}((s_{\beta_2}s_{\beta_1}s_{\beta_3}s_{\beta_2})(s_{\beta_4}s_{\beta_3}s_{\beta_5}s_{\beta_4})\cdots (s_{\beta_{i-2}}s_{\beta_{i-3}}s_{\beta_{i-1}}s_{\beta_{i-2}})s_{\delta-\rho})$$
$$\cdots((s_{\beta_2}s_{\beta_1}s_{\beta_3}s_{\beta_2})(s_{\beta_4}s_{\beta_3}s_{\beta_5}s_{\beta_4})s_{\delta-\rho}))((s_{\beta_2}s_{\beta_1}s_{\beta_3}s_{\beta_2})s_{\delta-\rho}))$$

Let $i$ be odd and $1<i\leq n-2$
$$t_{4\beta_1+8\beta_2+\cdots+4i\beta_i+4i\beta_{i+1}+\cdots 4i\beta_{n-2}+2i\beta_{n-1}+2i\beta_n}=$$
$$w_0^{W_{S\backslash \{s_{\beta_i}\}}}w_0s_{\delta-\rho}(s_{\beta_2}s_{\beta_3}\cdots s_{\beta_n}s_{\beta_1}s_{\beta_2}\cdots s_{\beta_n}(s_{\beta_{n-3}}s_{\beta_{n-2}})(s_{\beta_{n-4}}s_{\beta_{n-3}})\cdots$$ $$(s_{\beta_{i-1}}s_{\beta_i})s_{\delta-\rho})^{\frac{i-1}{2}}$$
$$(s_{\beta_2}s_{\beta_3}\cdots s_{\beta_n}s_{\beta_{n-2}}s_{\beta_{n-3}}\cdots s_{\beta_i}s_{\beta_1}s_{\beta_2}\cdots s_{\beta_{i-1}}s_{\delta-\rho})$$
$$((s_{\beta_2}s_{\beta_1}s_{\beta_3}s_{\beta_2})(s_{\beta_4}s_{\beta_3}s_{\beta_5}s_{\beta_4})\cdots (s_{\beta_{i-3}}s_{\beta_{i-4}}s_{\beta_{i-2}}s_{\beta_{i-3}})s_{\delta-\rho})$$
$$\cdots((s_{\beta_2}s_{\beta_1}s_{\beta_3}s_{\beta_2})(s_{\beta_4}s_{\beta_3}s_{\beta_5}s_{\beta_4})s_{\delta-\rho}))((s_{\beta_2}s_{\beta_1}s_{\beta_3}s_{\beta_2})s_{\delta-\rho}))$$

For $n$ even
$$t_{2\beta_1+4\beta_2+\cdots+ 2(n-2)\beta_{n-2}+n\beta_{n-1}+(n-2)\beta_n}=$$
$$w_0^{W_{S\backslash \{s_{\beta_{n-1}}\}}}w_0s_{\delta-\rho}s_{\beta_2}\cdots s_{\beta_{n-2}}s_{\beta_n}s_{\beta_1}s_{\beta_2}\cdots s_{\beta_{n-2}}s_{\delta-\rho}$$
$$((s_{\beta_2}s_{\beta_1}s_{\beta_3}s_{\beta_2})(s_{\beta_4}s_{\beta_3}s_{\beta_5}s_{\beta_4})\cdots (s_{\beta_{n-4}}s_{\beta_{n-5}}s_{\beta_{n-3}}s_{\beta_{n-4}})s_{\delta-\rho})$$
$$\cdots((s_{\beta_2}s_{\beta_1}s_{\beta_3}s_{\beta_2})(s_{\beta_4}s_{\beta_3}s_{\beta_5}s_{\beta_4})s_{\delta-\rho}))((s_{\beta_2}s_{\beta_1}s_{\beta_3}s_{\beta_2})s_{\delta-\rho}))$$

$$t_{2\beta_1+4\beta_2+\cdots+ 2(n-2)\beta_{n-2}+(n-2)\beta_{n-1}+n\beta_n}=$$
$$w_0^{W_{S\backslash \{s_{\beta_{n}}\}}}w_0s_{\delta-\rho}s_{\beta_2}\cdots s_{\beta_{n-2}}s_{\beta_{n-1}}s_{\beta_1}s_{\beta_2}\cdots s_{\beta_{n-2}}s_{\delta-\rho}$$
$$((s_{\beta_2}s_{\beta_1}s_{\beta_3}s_{\beta_2})(s_{\beta_4}s_{\beta_3}s_{\beta_5}s_{\beta_4})\cdots (s_{\beta_{n-4}}s_{\beta_{n-5}}s_{\beta_{n-3}}s_{\beta_{n-4}})s_{\delta-\rho})$$
$$\cdots((s_{\beta_2}s_{\beta_1}s_{\beta_3}s_{\beta_2})(s_{\beta_4}s_{\beta_3}s_{\beta_5}s_{\beta_4})s_{\delta-\rho}))((s_{\beta_2}s_{\beta_1}s_{\beta_3}s_{\beta_2})s_{\delta-\rho}))$$

For $n$ odd
$$t_{4\beta_1+8\beta_2+\cdots+(4n-8)\beta_{n-2}+2n\beta_{n-1}+(2n-4)\beta_{n}}=$$
$$w_0^{W_{S\backslash \{s_{\beta_{n-1}}\}}}w_0s_{\delta-\rho}(s_{\beta_2}s_{\beta_3}\cdots s_{\beta_{n-1}}s_{\beta_1}s_{\beta_2}\cdots s_{\beta_{n-2}}s_{\beta_n}s_{\delta-\rho})^{\frac{n+1}{2}}$$
$$((s_{\beta_2}s_{\beta_1}s_{\beta_3}s_{\beta_2})(s_{\beta_4}s_{\beta_3}s_{\beta_5}s_{\beta_4})\cdots (s_{\beta_{n-3}}s_{\beta_{n-4}}s_{\beta_{n-2}}s_{\beta_{n-3}})s_{\delta-\rho})$$
$$\cdots((s_{\beta_2}s_{\beta_1}s_{\beta_3}s_{\beta_2})(s_{\beta_4}s_{\beta_3}s_{\beta_5}s_{\beta_4})s_{\delta-\rho}))((s_{\beta_2}s_{\beta_1}s_{\beta_3}s_{\beta_2})s_{\delta-\rho}))$$

$$t_{4\beta_1+8\beta_2+\cdots+(4n-8)\beta_{n-2}+(2n-4)\beta_{n-1}+2n\beta_{n}}=$$
$$w_0^{W_{S\backslash \{s_{\beta_{n}}\}}}w_0s_{\delta-\rho}(s_{\beta_2}s_{\beta_3}\cdots s_{\beta_{n-2}}s_{\beta_{n}}s_{\beta_1}s_{\beta_2}\cdots s_{\beta_{n-2}}s_{\beta_{n-1}}s_{\delta-\rho})^{\frac{n+1}{2}}$$
$$((s_{\beta_2}s_{\beta_1}s_{\beta_3}s_{\beta_2})(s_{\beta_4}s_{\beta_3}s_{\beta_5}s_{\beta_4})\cdots (s_{\beta_{n-3}}s_{\beta_{n-4}}s_{\beta_{n-2}}s_{\beta_{n-3}})s_{\delta-\rho})$$
$$\cdots((s_{\beta_2}s_{\beta_1}s_{\beta_3}s_{\beta_2})(s_{\beta_4}s_{\beta_3}s_{\beta_5}s_{\beta_4})s_{\delta-\rho}))((s_{\beta_2}s_{\beta_1}s_{\beta_3}s_{\beta_2})s_{\delta-\rho}))$$

(2) Replace $w_0^{W_{S\backslash \{s_{\beta_i}\}}}w_0$ with a reduced expression of it, the expressions in (1) give  reduced expressions of these ``translations''.
\end{lemma}

\subsection{$E_6$}

The simple roots are $\beta_1,\beta_2,\cdots, \beta_6$ and they are numbered as in Section \ref{SectMain}. Let $\epsilon_1, \epsilon_2,\cdots, \epsilon_8$ be the standard basis of $\mathbb{R}^8$. Then simple roots can be realized as follow: $\beta_1=\frac{\epsilon_1-\epsilon_2-\epsilon_3-\cdots-\epsilon_7+\epsilon_8}{2\sqrt{2}}, \beta_2=\frac{\epsilon_1+\epsilon_2}{\sqrt{2}}, \beta_3=\frac{\epsilon_2-\epsilon_1}{\sqrt{2}}, \beta_4=\frac{\epsilon_3-\epsilon_2}{\sqrt{2}}, \beta_5=\frac{\epsilon_4-\epsilon_3}{\sqrt{2}}, \beta_6=\frac{\epsilon_5-\epsilon_4}{\sqrt{2}}$. Therefore the coroots are $2\beta_1,2\beta_2,\cdots, 2\beta_6.$ The positive roots are:
$\frac{\pm\epsilon_i+\epsilon_j}{\sqrt{2}}, 1\leq i<j\leq 5, \frac{\epsilon_8-\epsilon_7-\epsilon_6+\sum_{i=1}^5(-1)^{\mu(i)}\epsilon_i}{2\sqrt{2}}$ where $\mu(i)\in \mathbb{Z}_{>0}$ and $\sum_{i=1}^5\mu(i)$ is even. The highest root is $\beta_1+2(\beta_2+\beta_3)+3\beta_{4}+2\beta_5+\beta_6$.

\begin{lemma}
(1) A set of $i$-th distinguished ``translations'' $1\leq i\leq 6$ is given as follow.
 $$t_{8\beta_1+6\beta_2+10\beta_3+12\beta_4+8\beta_5+4\beta_6}$$$$=w_0^{W_{S\backslash \{s_{\beta_{1}}\}}}w_0s_{\delta-\rho}(s_{\beta_2}s_{\beta_4}s_{\beta_5}s_{\beta_3}s_{\beta_4}s_{\beta_1}s_{\beta_2}s_{\beta_3}s_{\beta_4}s_{\beta_5} $$ $$s_{\beta_6}s_{\delta-\rho})^2s_{\beta_2}s_{\beta_4}s_{\beta_5}s_{\beta_3}s_{\beta_4}s_{\beta_2}s_{\delta-\rho}$$

$$t_{2\beta_1+4\beta_2+4\beta_3+6\beta_4+4\beta_5+2\beta_6}$$
$$=w_0^{W_{S\backslash \{s_{\beta_{2}}\}}}w_0s_{\delta-\rho}$$

$$t_{10\beta_1+12\beta_2+20\beta_3+24\beta_4+16\beta_5+8\beta_6}$$
$$=w_0^{W_{S\backslash \{s_{\beta_{3}}\}}}w_0s_{\delta-\rho}(s_{\beta_2}s_{\beta_4}s_{\beta_5}s_{\beta_6}s_{\beta_3}s_{\beta_4}s_{\beta_5}s_{\beta_2}s_{\beta_4}s_{\beta_1}s_{\beta_3}s_{\beta_4}s_{\beta_2}s_{\beta_5}s_{\beta_4}s_{\beta_6}s_{\beta_5}s_{\delta-\rho})^2$$
$$(s_{\beta_2}s_{\beta_4}s_{\beta_3}s_{\beta_5}s_{\beta_4}s_{\beta_2}s_{\beta_1}s_{\beta_3}s_{\beta_4}s_{\beta_6}s_{\beta_5}s_{\beta_4}s_{\beta_3}s_{\beta_2}s_{\beta_4}s_{\delta-\rho})$$
$$(s_{\beta_2}s_{\beta_4}s_{\beta_3}s_{\beta_5}s_{\beta_6}s_{\beta_4}s_{\beta_1}s_{\beta_5}s_{\beta_3}s_{\beta_4}s_{\beta_2}s_{\delta-\rho})$$

$$t_{4\beta_1+6\beta_2+8\beta_3+12\beta_4+8\beta_5+4\beta_6}$$
$$=w_0^{W_{S\backslash \{s_{\beta_{4}}\}}}w_0s_{\delta-\rho}s_{\beta_2}s_{\beta_4}s_{\beta_3}s_{\beta_5}s_{\beta_6}s_{\beta_4}s_{\beta_1}s_{\beta_5}s_{\beta_3}s_{\beta_4}s_{\beta_2}s_{\delta-\rho}$$

$$t_{8\beta_1+12\beta_2+16\beta_3+24\beta_4+20\beta_5+10\beta_6}$$
$$=w_0^{W_{S\backslash \{s_{\beta_{5}}\}}}w_0s_{\delta-\rho}(s_{\beta_2}s_{\beta_4}s_{\beta_3}s_{\beta_1}s_{\beta_5}s_{\beta_4}s_{\beta_3}s_{\beta_2}s_{\beta_4}s_{\beta_6}s_{\beta_5}s_{\beta_4}s_{\beta_2}s_{\beta_3}s_{\beta_4}s_{\beta_1}s_{\beta_3}s_{\delta-\rho})^2$$
$$(s_{\beta_2}s_{\beta_4}s_{\beta_5}s_{\beta_3}s_{\beta_4}s_{\beta_2}s_{\beta_6}s_{\beta_5}s_{\beta_4}s_{\beta_1}s_{\beta_3}s_{\beta_4}s_{\beta_5}s_{\beta_2}s_{\beta_4}s_{\delta-\rho})$$
$$(s_{\beta_2}s_{\beta_4}s_{\beta_3}s_{\beta_5}s_{\beta_6}s_{\beta_4}s_{\beta_1}s_{\beta_5}s_{\beta_3}s_{\beta_4}s_{\beta_2}s_{\delta-\rho})$$

$$t_{4\beta_1+6\beta_2+8\beta_3+12\beta_4+10\beta_5+8\beta_6}$$$$=w_0^{W_{S\backslash \{s_{\beta_{6}}\}}}w_0s_{\delta-\rho}(s_{\beta_2}s_{\beta_4}s_{\beta_3}s_{\beta_5}s_{\beta_4}s_{\beta_6}s_{\beta_2}s_{\beta_5}s_{\beta_4}s_{\beta_3}$$ $$s_{\beta_1}s_{\delta-\rho})^2s_{\beta_2}s_{\beta_4}s_{\beta_3}s_{\beta_5}s_{\beta_4}s_{\beta_2}s_{\delta-\rho}$$

(2) Replace $w_0^{W_{S\backslash \{s_{\beta_i}\}}}w_0$ with a reduced expression of it, the expressions in (1) give the reduced expressions of these ``translations''.
\end{lemma}

\subsection{$E_7$}

The simple roots are $\beta_1,\beta_2,\cdots, \beta_7$ and they are numbered as in Section \ref{SectMain}. Let $\epsilon_1, \epsilon_2,\cdots, \epsilon_8$ be the standard basis of $\mathbb{R}^8$. Then simple roots can be realized as follow: $\beta_1=\frac{\epsilon_1-\epsilon_2-\epsilon_3-\cdots-\epsilon_7+\epsilon_8}{2\sqrt{2}}, \beta_2=\frac{\epsilon_1+\epsilon_2}{\sqrt{2}}, \beta_3=\frac{\epsilon_2-\epsilon_1}{\sqrt{2}}, \beta_4=\frac{\epsilon_3-\epsilon_2}{\sqrt{2}}, \beta_5=\frac{\epsilon_4-\epsilon_3}{\sqrt{2}}, \beta_6=\frac{\epsilon_5-\epsilon_4}{\sqrt{2}}, \beta_7=\frac{\epsilon_6-\epsilon_5}{\sqrt{2}}$. Therefore the coroots are $2\beta_1,2\beta_2,\cdots, 2\beta_7.$ The positive roots are: $\epsilon_8-\epsilon_7$,
$\frac{\pm\epsilon_i+\epsilon_j}{\sqrt{2}}, 1\leq i<j\leq 6, \frac{\epsilon_8-\epsilon_7+\sum_{i=1}^6(-1)^{\mu(i)}\epsilon_i}{2\sqrt{2}}$ where $\mu(i)\in \mathbb{Z}_{>0}$ and $\sum_{i=1}^6\mu(i)$ is odd. The highest root is $2\beta_1+2\beta_2+3\beta_3+4\beta_4+3\beta_{5}+2\beta_6+\beta_7$.

\begin{lemma}

(1) A set of $i$-th distinguished ``translations'' $1\leq i\leq 7$ is given as follow.
$$t_{4\beta_1+4\beta_2+6\beta_3+8\beta_4+6\beta_5+4\beta_6+2\beta_7}$$
$$=w_0^{W_{S\backslash \{s_{\beta_{1}}\}}}w_0s_{\delta-\rho}$$

 $$t_{8\beta_1+14\beta_2+16\beta_3+24\beta_4+18\beta_5+12\beta_6+6\beta_7}$$
$$=w_0^{W_{S\backslash \{s_{\beta_{2}}\}}}w_0s_{\delta-\rho}s_{\beta_1}s_{\beta_3}s_{\beta_4}s_{\beta_2}s_{\beta_5}s_{\beta_4}s_{\beta_3}s_{\beta_1}s_{\beta_6}s_{\beta_5}s_{\beta_4}s_{\beta_2}s_{\beta_3}s_{\beta_4}s_{\beta_5}$$
$$s_{\beta_7}s_{\beta_6}s_{\beta_5}s_{\beta_4}s_{\beta_3}s_{\beta_2}s_{\beta_1}s_{\beta_4}s_{\beta_3}s_{\beta_5}s_{\beta_4}s_{\beta_2}s_{\delta-\rho}s_{\beta_1}s_{\beta_3}s_{\beta_4}s_{\beta_2}s_{\beta_5}s_{\beta_4}s_{\beta_3}s_{\beta_1}s_{\beta_6}s_{\beta_5}$$
$$s_{\beta_4}s_{\beta_2}s_{\beta_3}s_{\beta_4}s_{\beta_7}s_{\beta_5}s_{\beta_6}s_{\delta-\rho}s_{\beta_1}s_{\beta_3}s_{\beta_4}s_{\beta_5}s_{\beta_2}s_{\beta_4}s_{\beta_3}s_{\beta_1}s_{\delta-\rho}$$

 $$t_{6\beta_1+8\beta_2+12\beta_3+16\beta_4+12\beta_5+8\beta_6+4\beta_7}$$
$$=w_0^{W_{S\backslash \{s_{\beta_{3}}\}}}w_0s_{\delta-\rho}s_{\beta_1}s_{\beta_3}s_{\beta_4}s_{\beta_5}s_{\beta_2}s_{\beta_6}s_{\beta_4}s_{\beta_7}s_{\beta_5}s_{\beta_3}s_{\beta_6}s_{\beta_4}s_{\beta_5}s_{\beta_2}s_{\beta_4}s_{\beta_3}s_{\beta_1}s_{\delta-\rho}$$

 $$t_{8\beta_1+12\beta_2+16\beta_3+24\beta_4+18\beta_5+12\beta_6+6\beta_7}$$
$$=w_0^{W_{S\backslash \{s_{\beta_{4}}\}}}w_0s_{\delta-\rho}s_{\beta_1}s_{\beta_3}s_{\beta_4}s_{\beta_2}s_{\beta_5}s_{\beta_4}s_{\beta_3}s_{\beta_1}s_{\beta_6}s_{\beta_5}s_{\beta_4}s_{\beta_3}s_{\beta_2}s_{\beta_4}s_{\beta_7}$$
$$s_{\beta_6}s_{\beta_5}s_{\beta_4}s_{\beta_2}s_{\beta_3}s_{\beta_4}s_{\beta_1}s_{\beta_3}s_{\delta-\rho}s_{\beta_1}s_{\beta_3}s_{\beta_4}s_{\beta_5}s_{\beta_2}$$
$$s_{\beta_6}s_{\beta_4}s_{\beta_7}s_{\beta_5}s_{\beta_3}s_{\beta_6}s_{\beta_4}s_{\beta_5}s_{\beta_2}s_{\beta_4}s_{\beta_3}s_{\beta_1}s_{\delta-\rho}$$

 $$t_{12\beta_1+18\beta_2+24\beta_3+36\beta_4+30\beta_5+20\beta_6+10\beta_7}$$
$$=w_0^{W_{S\backslash \{s_{\beta_{5}}\}}}w_0s_{\delta-\rho}s_{\beta_1}s_{\beta_3}s_{\beta_4}s_{\beta_5}s_{\beta_6}s_{\beta_7}s_{\beta_2}s_{\beta_4}s_{\beta_3}s_{\beta_5}s_{\beta_1}s_{\beta_4}s_{\beta_2}s_{\beta_6}$$
$$s_{\beta_3}s_{\beta_4}s_{\beta_5}s_{\beta_6}s_{\beta_7}s_{\beta_4}s_{\beta_3}s_{\beta_1}s_{\beta_2}s_{\beta_4}s_{\beta_5}s_{\beta_3}s_{\beta_6}s_{\beta_4}s_{\beta_5}s_{\delta-\rho}s_{\beta_1}s_{\beta_3}$$
$$s_{\beta_4}s_{\beta_5}s_{\beta_6}s_{\beta_7}s_{\beta_2}s_{\beta_4}s_{\beta_3}s_{\beta_5}s_{\beta_1}s_{\beta_4}s_{\beta_2}s_{\beta_6}s_{\beta_3}s_{\beta_4}s_{\beta_5}s_{\beta_6}s_{\beta_4}s_{\beta_3}s_{\beta_1}$$
$$s_{\beta_2}s_{\beta_4}s_{\beta_5}s_{\beta_3}s_{\beta_4}s_{\delta-\rho}s_{\beta_1}s_{\beta_3}s_{\beta_4}s_{\beta_2}s_{\beta_5}s_{\beta_4}s_{\beta_3}s_{\beta_1}s_{\beta_6}s_{\beta_5}s_{\beta_4}s_{\beta_3}s_{\beta_2}$$
$$s_{\beta_4}s_{\beta_7}s_{\beta_6}s_{\beta_5}s_{\beta_4}s_{\beta_2}s_{\beta_3}s_{\beta_4}s_{\beta_1}s_{\beta_3}s_{\delta-\rho}s_{\beta_1}s_{\beta_3}s_{\beta_4}s_{\beta_5}s_{\beta_2}s_{\beta_6}s_{\beta_4}s_{\beta_7}$$
$$s_{\beta_5}s_{\beta_3}s_{\beta_6}s_{\beta_4}s_{\beta_5}s_{\beta_2}s_{\beta_4}s_{\beta_3}s_{\beta_1}s_{\delta-\rho}$$

 $$t_{4\beta_1+6\beta_2+8\beta_3+12\beta_4+10\beta_5+8\beta_6+4\beta_7}$$
$$=w_0^{W_{S\backslash \{s_{\beta_{6}}\}}}w_0s_{\delta-\rho}s_{\beta_1}s_{\beta_3}s_{\beta_4}s_{\beta_5}s_{\beta_2}s_{\beta_4}s_{\beta_3}s_{\beta_1}s_{\delta-\rho}$$

$$t_{4\beta_1+6\beta_2+8\beta_3+12\beta_4+10\beta_5+8\beta_6+6\beta_7}$$
$$=w_0^{W_{S\backslash \{s_{\beta_{7}}\}}}w_0s_{\delta-\rho}s_{\beta_1}s_{\beta_3}s_{\beta_4}s_{\beta_2}s_{\beta_5}s_{\beta_6}s_{\beta_4}s_{\beta_5}s_{\beta_3}s_{\beta_4}s_{\beta_1}s_{\beta_2}s_{\beta_3}s_{\beta_4}s_{\beta_5}$$
$$s_{\beta_6}s_{\delta-\rho}s_{\beta_1}s_{\beta_3}s_{\beta_4}s_{\beta_5}s_{\beta_2}s_{\beta_4}s_{\beta_3}s_{\beta_1}s_{\delta-\rho}$$

(2) Replace $w_0^{W_{S\backslash \{s_{\beta_i}\}}}w_0$ with a reduced expression of it, the expressions in (1) give the reduced expressions of these ``translations''.
\end{lemma}

\subsection{$E_8$}

The simple roots are $\beta_1,\beta_2,\cdots, \beta_8$ and they are numbered as in Section \ref{SectMain}. Let $\epsilon_1, \epsilon_2,\cdots, \epsilon_8$ be the standard basis of $\mathbb{R}^8$. Then simple roots can be realized as follow: $\beta_1=\frac{\epsilon_1-\epsilon_2-\epsilon_3-\cdots-\epsilon_7+\epsilon_8}{2\sqrt{2}}, \beta_2=\frac{\epsilon_1+\epsilon_2}{\sqrt{2}}, \beta_3=\frac{\epsilon_2-\epsilon_1}{\sqrt{2}}, \beta_4=\frac{\epsilon_3-\epsilon_2}{\sqrt{2}}, \beta_5=\frac{\epsilon_4-\epsilon_3}{\sqrt{2}}, \beta_6=\frac{\epsilon_5-\epsilon_4}{\sqrt{2}}, \beta_7=\frac{\epsilon_6-\epsilon_5}{\sqrt{2}}, \beta_8=\frac{\epsilon_7-\epsilon_6}{\sqrt{2}}$. Therefore the coroots are $2\beta_1,2\beta_2,\cdots, 2\beta_8.$ The positive roots are:
$\frac{\pm\epsilon_i+\epsilon_j}{\sqrt{2}}, 1\leq i<j\leq 8, \frac{\epsilon_8+\sum_{i=1}^7(-1)^{\mu(i)}\epsilon_i}{2\sqrt{2}}$ where $\mu(i)\in \mathbb{Z}_{>0}$ and $\sum_{i=1}^7\mu(i)$ is even. The highest root is $2\beta_1+3\beta_2+4\beta_3+6\beta_4+5\beta_{5}+4\beta_6+3\beta_7+2\beta_8$.

\begin{lemma}
(1) A set of $i$-th distinguished ``translations'' $1\leq i\leq 8$ is given as follow.
$$t_{8\beta_1+10\beta_2+14\beta_3+20\beta_4+16\beta_5+12\beta_6+8\beta_7+4\beta_8}$$
$$=w_0^{W_{S\backslash \{s_{\beta_{1}}\}}}w_0s_{\delta-\rho}s_{\beta_8}s_{\beta_7}s_{\beta_6}s_{\beta_5}s_{\beta_4}s_{\beta_3}s_{\beta_2}s_{\beta_4}s_{\beta_5}s_{\beta_6}s_{\beta_7}s_{\beta_8}s_{\delta-\rho}$$

$$t_{10\beta_1+16\beta_2+20\beta_3+30\beta_4+24\beta_5+18\beta_6+12\beta_7+6\beta_8}$$

$$=w_0^{W_{S\backslash \{s_{\beta_{2}}\}}}w_0s_{\delta-\rho}s_{\beta_8}s_{\beta_7}s_{\beta_6}s_{\beta_5}s_{\beta_4}s_{\beta_2}s_{\beta_3}s_{\beta_4}s_{\beta_5}s_{\beta_6}s_{\beta_7}s_{\beta_8}s_{\beta_1}s_{\beta_3}s_{\beta_4}s_{\beta_5}s_{\beta_2}$$
$$s_{\beta_6}s_{\beta_4}s_{\beta_7}
s_{\beta_5}s_{\beta_3}s_{\beta_6}s_{\beta_4}s_{\beta_5}s_{\beta_2}s_{\beta_4}s_{\beta_3}s_{\beta_1}s_{\delta-\rho}s_{\beta_8}s_{\beta_7}s_{\beta_6}s_{\beta_5}s_{\beta_4}s_{\beta_3}s_{\beta_2}s_{\beta_4}s_{\beta_5}s_{\beta_6}$$
$$s_{\beta_7}s_{\beta_8}s_{\delta-\rho}$$

$$t_{14\beta_1+20\beta_2+28\beta_3+40\beta_4+32\beta_5+24\beta_6+16\beta_7+8\beta_8}$$
$$=w_0^{W_{S\backslash \{s_{\beta_{3}}\}}}w_0s_{\delta-\rho}s_{\beta_8}s_{\beta_7}s_{\beta_6}s_{\beta_5}s_{\beta_4}s_{\beta_3}s_{\beta_1}s_{\beta_2}s_{\beta_4}s_{\beta_3}s_{\beta_5}s_{\beta_4}s_{\beta_2}s_{\beta_6}s_{\beta_5}s_{\beta_4}s_{\beta_3}$$ $$s_{\beta_1}s_{\beta_7}s_{\beta_6}s_{\beta_5}s_{\beta_4}s_{\beta_2}s_{\beta_3}s_{\beta_4}s_{\beta_5}s_{\beta_8}s_{\beta_7}s_{\beta_6}s_{\beta_5}s_{\beta_4}s_{\beta_3}s_{\beta_2}s_{\beta_1}s_{\beta_4}s_{\beta_3}s_{\beta_5}s_{\beta_4}s_{\beta_2}s_{\delta-\rho}$$
$$s_{\beta_8}s_{\beta_7}s_{\beta_6}s_{\beta_5}s_{\beta_4}s_{\beta_2}s_{\beta_3}s_{\beta_4}s_{\beta_5}s_{\beta_6}s_{\beta_7}s_{\beta_8}s_{\beta_1}s_{\beta_3}s_{\beta_4}s_{\beta_5}s_{\beta_2}s_{\beta_6}s_{\beta_4}s_{\beta_7}s_{\beta_5}s_{\beta_3}s_{\beta_6}s_{\beta_4}$$
$$s_{\beta_5}s_{\beta_2}s_{\beta_4}s_{\beta_3}s_{\beta_1}s_{\delta-\rho}s_{\beta_8}s_{\beta_7}s_{\beta_6}s_{\beta_5}s_{\beta_4}s_{\beta_3}s_{\beta_2}s_{\beta_4}s_{\beta_5}s_{\beta_6}s_{\beta_7}s_{\beta_8}s_{\delta-\rho}$$

$$t_{20\beta_1+30\beta_2+40\beta_3+60\beta_4+48\beta_5+36\beta_6+24\beta_7+12\beta_8}$$
$$=w_0^{W_{S\backslash \{s_{\beta_{4}}\}}}w_0s_{\delta-\rho}s_{\beta_8}s_{\beta_7}s_{\beta_6}s_{\beta_5}s_{\beta_4}s_{\beta_3}s_{\beta_2}s_{\beta_4}s_{\beta_5}s_{\beta_6}s_{\beta_7}s_{\beta_8}s_{\beta_1}s_{\beta_3}s_{\beta_4}s_{\beta_2}s_{\beta_5}s_{\beta_4}s_{\beta_3}s_{\beta_1}$$
$$s_{\beta_6}s_{\beta_5}s_{\beta_4}s_{\beta_3}s_{\beta_2}s_{\beta_4}s_{\beta_5}s_{\beta_6}s_{\beta_7}s_{\beta_8}s_{\beta_6}s_{\beta_5}s_{\beta_4}s_{\beta_2}s_{\beta_3}s_{\beta_4}s_{\beta_5}s_{\beta_6}s_{\beta_7}s_{\beta_1}s_{\beta_3}s_{\beta_4}s_{\beta_5}s_{\beta_2}s_{\beta_6}s_{\beta_4}$$
$$s_{\beta_5}s_{\delta-\rho}s_{\beta_8}s_{\beta_7}s_{\beta_6}s_{\beta_5}s_{\beta_4}s_{\beta_2}s_{\beta_3}s_{\beta_4}s_{\beta_1}s_{\beta_3}s_{\beta_5}s_{\beta_6}s_{\beta_7}s_{\beta_8}s_{\beta_4}s_{\beta_5}s_{\beta_6}s_{\beta_7}s_{\beta_2}s_{\beta_4}s_{\beta_3}s_{\beta_5}s_{\beta_1}$$
$$s_{\beta_4}s_{\beta_2}s_{\beta_6}s_{\beta_3}s_{\beta_4}s_{\beta_5}s_{\beta_6}s_{\beta_7}s_{\beta_8}s_{\beta_4}s_{\beta_3}s_{\beta_1}s_{\beta_2}s_{\beta_4}s_{\beta_5}s_{\beta_3}s_{\beta_6}s_{\beta_4}s_{\beta_7}s_{\beta_5}s_{\beta_6}s_{\delta-\rho}s_{\beta_8}s_{\beta_7}s_{\beta_6}$$
$$s_{\beta_5}s_{\beta_4}s_{\beta_3}s_{\beta_2}s_{\beta_1}s_{\beta_4}s_{\beta_3}s_{\beta_5}s_{\beta_4}s_{\beta_6}s_{\beta_7}s_{\beta_8}s_{\beta_5}s_{\beta_6}s_{\beta_7}s_{\beta_2}s_{\beta_4}s_{\beta_5}s_{\beta_6}s_{\beta_3}s_{\beta_4}s_{\beta_2}s_{\beta_5}s_{\beta_4}s_{\beta_1}s_{\beta_3}$$
$$s_{\beta_4}s_{\beta_5}s_{\beta_2}s_{\beta_6}s_{\beta_4}s_{\beta_7}s_{\beta_5}s_{\beta_8}s_{\beta_6}s_{\beta_7}s_{\delta-\rho}s_{\beta_8}s_{\beta_7}s_{\beta_6}s_{\beta_5}s_{\beta_4}s_{\beta_2}s_{\beta_3}s_{\beta_4}s_{\beta_1}s_{\beta_5}s_{\beta_3}s_{\beta_6}s_{\beta_4}s_{\beta_7}$$
$$s_{\beta_5}s_{\beta_2}s_{\beta_6}s_{\beta_4}s_{\beta_5}s_{\beta_3}s_{\beta_4}s_{\beta_1}s_{\beta_3}s_{\beta_2}s_{\beta_4}s_{\beta_5}s_{\beta_6}s_{\beta_7}s_{\beta_8}s_{\delta-\rho}$$

$$t_{16\beta_1+24\beta_2+32\beta_3+48\beta_4+40\beta_5+30\beta_6+20\beta_7+10\beta_8}$$
$$=w_0^{W_{S\backslash \{s_{\beta_{5}}\}}}w_0s_{\delta-\rho}s_{\beta_8}s_{\beta_7}s_{\beta_6}s_{\beta_5}s_{\beta_4}s_{\beta_2}s_{\beta_3}s_{\beta_4}s_{\beta_1}s_{\beta_3}s_{\beta_5}s_{\beta_6}s_{\beta_7}s_{\beta_8}s_{\beta_4}s_{\beta_5}s_{\beta_6}s_{\beta_7}s_{\beta_2}s_{\beta_4}$$
$$s_{\beta_3}s_{\beta_5}s_{\beta_1}
s_{\beta_4}s_{\beta_2}s_{\beta_6}s_{\beta_3}s_{\beta_4}s_{\beta_5}s_{\beta_6}s_{\beta_7}s_{\beta_8}s_{\beta_4}s_{\beta_3}s_{\beta_1}s_{\beta_2}s_{\beta_4}s_{\beta_5}s_{\beta_3}s_{\beta_6}s_{\beta_4}s_{\beta_7}s_{\beta_5}s_{\beta_6}s_{\delta-\rho}s_{\beta_8}$$
$$s_{\beta_7}s_{\beta_6}s_{\beta_5}
s_{\beta_4}s_{\beta_3}s_{\beta_2}s_{\beta_1}s_{\beta_4}s_{\beta_3}s_{\beta_5}s_{\beta_4}s_{\beta_6}s_{\beta_7}s_{\beta_8}s_{\beta_5}s_{\beta_6}s_{\beta_7}s_{\beta_2}s_{\beta_4}s_{\beta_5}s_{\beta_6}s_{\beta_3}s_{\beta_4}s_{\beta_2}s_{\beta_5}s_{\beta_4}s_{\beta_1}$$
$$s_{\beta_3}s_{\beta_4}s_{\beta_5}
s_{\beta_2}s_{\beta_6}s_{\beta_4}s_{\beta_7}s_{\beta_5}s_{\beta_8}s_{\beta_6}s_{\beta_7}s_{\delta-\rho}s_{\beta_8}s_{\beta_7}s_{\beta_6}s_{\beta_5}s_{\beta_4}s_{\beta_2}s_{\beta_3}s_{\beta_4}s_{\beta_1}s_{\beta_5}s_{\beta_3}s_{\beta_6}s_{\beta_4}$$$$s_{\beta_7}s_{\beta_5}
s_{\beta_2}s_{\beta_6}
s_{\beta_4}s_{\beta_5}s_{\beta_3}s_{\beta_4}s_{\beta_1}s_{\beta_3}s_{\beta_2}s_{\beta_4}s_{\beta_5}s_{\beta_6}s_{\beta_7}s_{\beta_8}s_{\delta-\rho}$$

$$t_{12\beta_1+18\beta_2+24\beta_3+36\beta_4+30\beta_5+24\beta_6+16\beta_7+8\beta_8}$$
$$=w_0^{W_{S\backslash \{s_{\beta_{6}}\}}}w_0s_{\delta-\rho}s_{\beta_8}s_{\beta_7}s_{\beta_6}s_{\beta_5}s_{\beta_4}s_{\beta_3}s_{\beta_2}s_{\beta_1}s_{\beta_4}s_{\beta_3}s_{\beta_5}s_{\beta_4}s_{\beta_6}s_{\beta_7}s_{\beta_8}s_{\beta_5}s_{\beta_6}s_{\beta_7}s_{\beta_2}s_{\beta_4}s_{\beta_5}$$$$s_{\beta_6}s_{\beta_3}s_{\beta_4}
s_{\beta_2}s_{\beta_5}s_{\beta_4}s_{\beta_1}s_{\beta_3}s_{\beta_4}s_{\beta_5}s_{\beta_2}s_{\beta_6}s_{\beta_4}s_{\beta_7}s_{\beta_5}s_{\beta_8}s_{\beta_6}s_{\beta_7}s_{\delta-\rho}s_{\beta_8}s_{\beta_7}s_{\beta_6}s_{\beta_5}s_{\beta_4}s_{\beta_2}$$$$s_{\beta_3}s_{\beta_4}s_{\beta_1}s_{\beta_5}s_{\beta_3}
s_{\beta_6}s_{\beta_4}s_{\beta_7}s_{\beta_5}s_{\beta_2}s_{\beta_6}s_{\beta_4}s_{\beta_5}s_{\beta_3}s_{\beta_4}s_{\beta_1}s_{\beta_3}s_{\beta_2}s_{\beta_4}s_{\beta_5}s_{\beta_6}s_{\beta_7}s_{\beta_8}s_{\delta-\rho}$$

$$t_{8\beta_1+12\beta_2+16\beta_3+24\beta_4+20\beta_5+16\beta_6+12\beta_7+6\beta_8}$$
$$=w_0^{W_{S\backslash \{s_{\beta_{7}}\}}}w_0s_{\delta-\rho}s_{\beta_8}s_{\beta_7}s_{\beta_6}s_{\beta_5}s_{\beta_4}s_{\beta_2}s_{\beta_3}s_{\beta_4}s_{\beta_1}s_{\beta_5}s_{\beta_3}s_{\beta_6}s_{\beta_4}s_{\beta_7}s_{\beta_5}s_{\beta_2}s_{\beta_6}s_{\beta_4}s_{\beta_5}s_{\beta_3}$$$$s_{\beta_4}s_{\beta_1}s_{\beta_3}
s_{\beta_2}s_{\beta_4}s_{\beta_5}s_{\beta_6}s_{\beta_7}s_{\beta_8}s_{\delta-\rho}$$

$$t_{4\beta_1+6\beta_2+8\beta_3+12\beta_4+10\beta_5+8\beta_6+6\beta_7+4\beta_8}$$
$$=w_0^{W_{S\backslash \{s_{\beta_{8}}\}}}w_0s_{\delta-\rho}$$

(2) Replace $w_0^{W_{S\backslash \{s_{\beta_i}\}}}w_0$ with a reduced expression of it, the expressions in (1) give the reduced expressions of these ``translations''.
\end{lemma}

\subsection{$F_4$}

The simple roots are denoted by $\beta_1,\beta_2,\beta_3, \beta_4$ and they are numbered as in Section \ref{SectMain}. Let $\epsilon_1, \epsilon_2,\cdots, \epsilon_4$ be the standard basis of $\mathbb{R}^4$. Then simple roots can be realized as follow: $\beta_1=\epsilon_2-\epsilon_3, \beta_2=\epsilon_3-\epsilon_4, \beta_3=\epsilon_4, \beta_4=\frac{\epsilon_1-\epsilon_2-\epsilon_3-\epsilon_4}{2}.$ Therefore the coroots are $\beta_1,\beta_2, 2\beta_3, 2\beta_4.$ The positive roots are:
$\epsilon_i, 1\leq i\leq 4, \epsilon_i\pm\epsilon_j 1\leq i<j\leq 4, \frac{\epsilon_1\pm\epsilon_2\pm \epsilon_3\pm \epsilon_4}{2}$. The highest root is $2\beta_1+3\beta_2+4\beta_3+2\beta_4$.

\begin{lemma}
(1) A set of $i$-th distinguished ``translations'' $1\leq i\leq 4$ is given as follow.
$$t_{2\beta_1+3\beta_2+4\beta_3+2\beta_4}$$
$$=w_0^{W_{S\backslash \{s_{\beta_{1}}\}}}w_0s_{\delta-\rho}.$$

$$t_{3\beta_1+6\beta_2+8\beta_3+4\beta_4}$$
$$=w_0^{W_{S\backslash \{s_{\beta_{2}}\}}}w_0s_{\delta-\rho}s_{\beta_1}s_{\beta_2}s_{\beta_3}s_{\beta_4}s_{\beta_2}s_{\beta_3}s_{\beta_2}s_{\beta_1}s_{\delta-\rho}.$$

$$t_{4\beta_1+8\beta_2+12\beta_3+6\beta_4}$$
$$=w_0^{W_{S\backslash \{s_{\beta_{3}}\}}}w_0s_{\delta-\rho}s_{\beta_1}s_{\beta_2}s_{\beta_3}s_{\beta_2}s_{\beta_1}s_{\beta_4}s_{\beta_3}s_{\beta_2}s_{\beta_3}s_{\beta_1}s_{\beta_2}s_{\delta-\rho}$$
$$s_{\beta_1}s_{\beta_2}s_{\beta_3}s_{\beta_4}s_{\beta_2}s_{\beta_3}s_{\beta_2}s_{\beta_1}s_{\delta-\rho}$$

$$t_{2\beta_1+4\beta_2+6\beta_3+4\beta_4}$$
$$=w_0^{W_{S\backslash \{s_{\beta_{4}}\}}}w_0s_{\delta-\rho}s_{\beta_1}s_{\beta_2}s_{\beta_3}s_{\beta_2}s_{\beta_1}s_{\delta-\rho}$$

(2) Replace $w_0^{W_{S\backslash \{s_{\beta_i}\}}}w_0$ with a reduced expression of it, the expressions in (1) give the reduced expressions of these ``translations''.
\end{lemma}

\subsection{$G_2$}

The simple roots are denoted by $\beta_1,\beta_2$ and they are numbered as in Section \ref{SectMain}. Let $\epsilon_1, \epsilon_2,\epsilon_3$ be the standard basis of $\mathbb{R}^3$. Then simple roots can be realized as follow: $\beta_1=\epsilon_1-\epsilon_2, \beta_2=-2\epsilon_1+\epsilon_2+\epsilon_3.$ Therefore the coroots are $\beta_1,\frac{\beta_2}{3}.$ The positive roots are:
$\beta_1,\beta_2,\beta_1+\beta_2, 2\beta_1+\beta_2, 3\beta_1+\beta_2,3\beta_1+2\beta_2$. The highest root is $3\beta_1+2\beta_2$.

\begin{lemma}
(1) A set of $i$-th distinguished ``translations'' $i=1, 2$ is given as follow.
$$t_{2\beta_1+\beta_2}$$
$$=w_0^{W_{S\backslash \{s_{\beta_{1}}\}}}w_0s_{\delta-\rho}s_{\beta_2}s_{\beta_1}s_{\beta_2}s_{\delta-\rho}$$

$$t_{\beta_1+\frac23\beta_2}$$
$$=w_0^{W_{S\backslash \{s_{\beta_{2}}\}}}w_0s_{\delta-\rho}$$

(2) Replace $w_0^{W_{S\backslash \{s_{\beta_i}\}}}w_0$ with a reduced expression of it, the expressions in (1) give the reduced expressions of these ``translations''.
\end{lemma}

\section{infinite Coxeter elements}\label{infiniteCox}

A Coxeter element of $(W,S)$ is a product of all simple reflections $s\in S$ in any given order. Coxeter elements are important in the theory of Coxeter groups. In particular they are used in the study of polynomial invariants of  Coxeter groups. By \cite{Shi}, there exists a bijection between the set of Coxeter elements and the set of acyclic orientation of the Coxeter graph of $(W,S)$. Let $c$ be a Coxeter element. It determines an acyclic orientation in the following way: Suppose that $s_i, s_j\in S$ are adjacent in the Coxeter graph. We assign the direction from $s_i$ to $s_j$ if $s_i$ occurs to the left of $s_j$ in a (equivalently any) reduced expression of $c$. If $W$ is infinite and irreducible, an infinite reduced word of the form $c^{\infty}$ where $c$ is a Coxeter element is called an infinite Coxeter element.

It is natural to ask whether an infinite Coxeter element is minimal and  which of the minimal infinite reduced words described in the main theorem correspond to infinite Coxeter elements. For type $\widetilde{C_n}$, this has a satisfactory answer which will be discussed in this section. Now let $W$ and $\Phi$ be of type $C_n.$

 Let $u\in \widetilde{W}$ and let $s$ be a simple reflection of $\widetilde{W}$. We say that $s$ is a left (resp. right) descent of $u$ if $\ell(su)<\ell(u)$ (resp. $\ell(us)<\ell(u)$). We define a $W$-action on the set of Coxeter elements of $\widetilde{W}$ as follow.

Let $c$ be a Coxeter element of $\widetilde{W}$. Define $s_{\beta_i}\cdot c=s_{\beta_i}cs_{\beta_i}$ if $s_{\beta_i}$ is a left descent or a right descent of $c$ and $s_{\beta_i}\cdot c=c$ otherwise.
Let $w=s_{\beta_{i_1}}\cdots s_{\beta_{i_k}}$. Define $w\cdot c=s_{\beta_{i_1}}\cdot (s_{\beta_{i_2}}\cdot \cdots (s_{\beta_{i_k}}\cdot c))$.

\begin{lemma}\label{actionB}
(1) This is a well-defined $W$-action.

(2) Under this action, the stabilizer of $s_{\beta_n}s_{\beta_{n-1}}\cdots s_{\beta_1}s_{\delta-\rho}$ is the parabolic subgroup $W_{\{s_{\beta_1}, \cdots, s_{\beta_{n-1}}\}}$.
\end{lemma}

\begin{proof}
(1) Each Coxeter element corresponds to the orientation of a path graph with $n+1$ nodes $O_c$. If $s_{\beta_i}$ is a left descent (resp. right descent) of $c$, then the $i$-th node is a source (resp. sink) of this orientation. If $s_{\beta_i}$ is a left descent or a right descent of $c$, for $i<n$ the action of $s_{\beta_i}$ on $O_c$ changes the directions of the two edges which have the $i$-th vertex as their  common endpoint and for $i=n$, the action changes the rightmost edge. Otherwise the action of  $s_{\beta_i}$ leaves $O_c$ unchanged.
First clearly $s_i\cdot (s_i\cdot c)=c$.
Now let $|i-j|\geq 2$. One easily verifies that $s_i\cdot (s_j\cdot c)=s_j\cdot (s_i\cdot c)$.

Now we verify that $s_{\beta_i} \cdot (s_{\beta_{i+1}}\cdot (s_{\beta_i}\cdot c))=s_{\beta_{i+1}} \cdot (s_{\beta_{i}}\cdot (s_{\beta_{i+1}}\cdot c)), 1\leq i\leq n-2$. It suffices to check the actions on the corresponding orientation $O_c.$

There are eight cases. We depict these cases as follows.

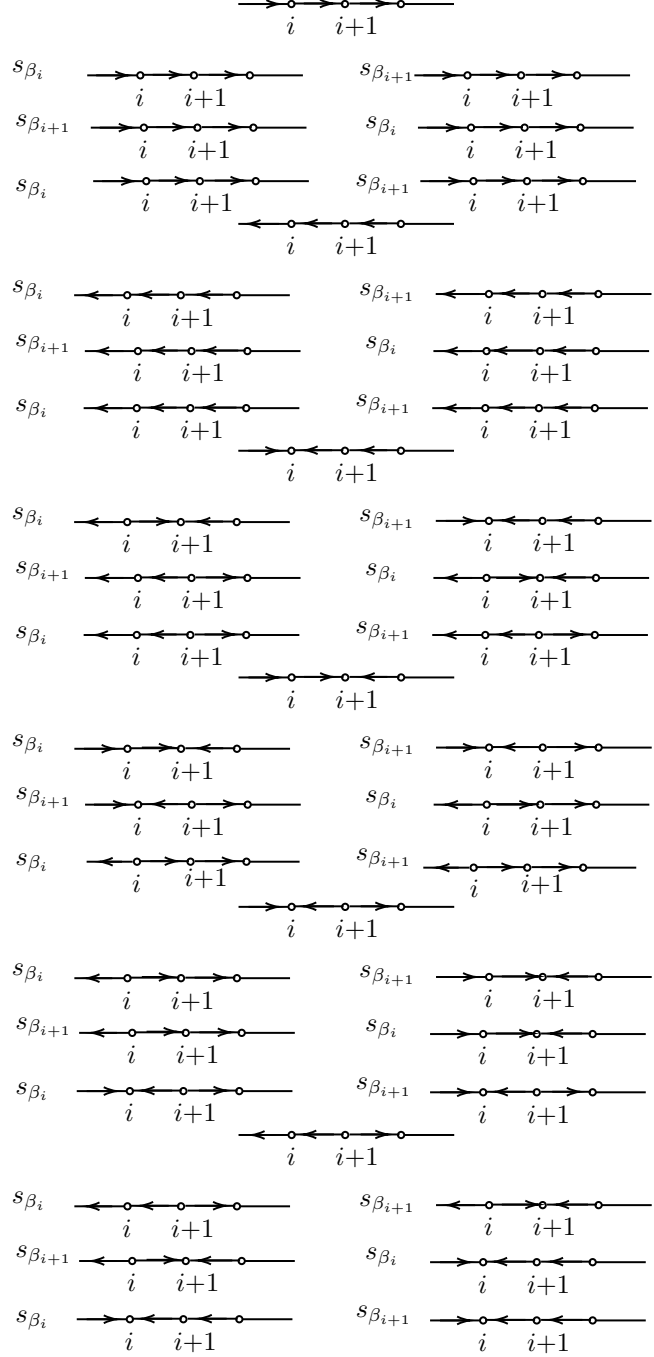
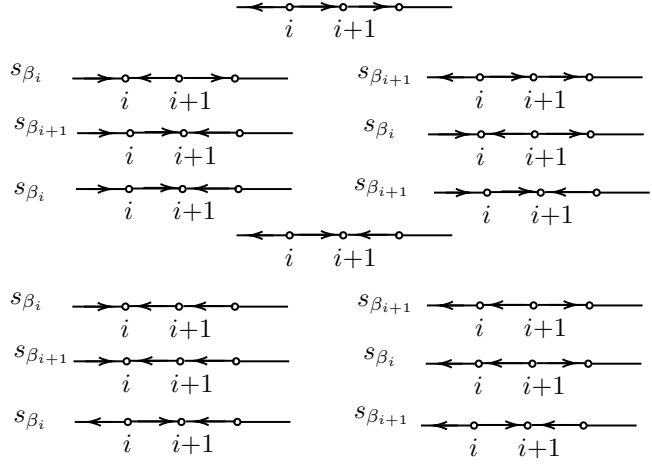
\tikzset{every picture/.style={line width=0.75pt}} 



Now we verify that $s_{\beta_{n-1}} \cdot (s_{\beta_{n}}\cdot (s_{\beta_{n-1}}\cdot (s_{\beta_{n}}\cdot c)))=s_{\beta_{n}} \cdot (s_{\beta_{n-1}}\cdot (s_{\beta_{n}}\cdot (s_{\beta_{n-1}}\cdot c)))$. It suffices to check the actions on the corresponding orientation $O_c.$

There are four cases. We depict these cases as follows.

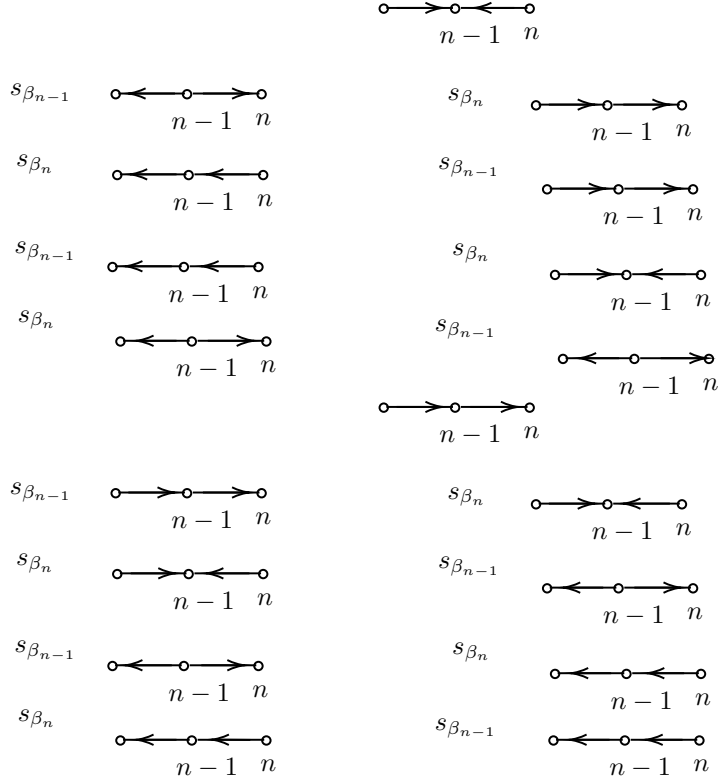
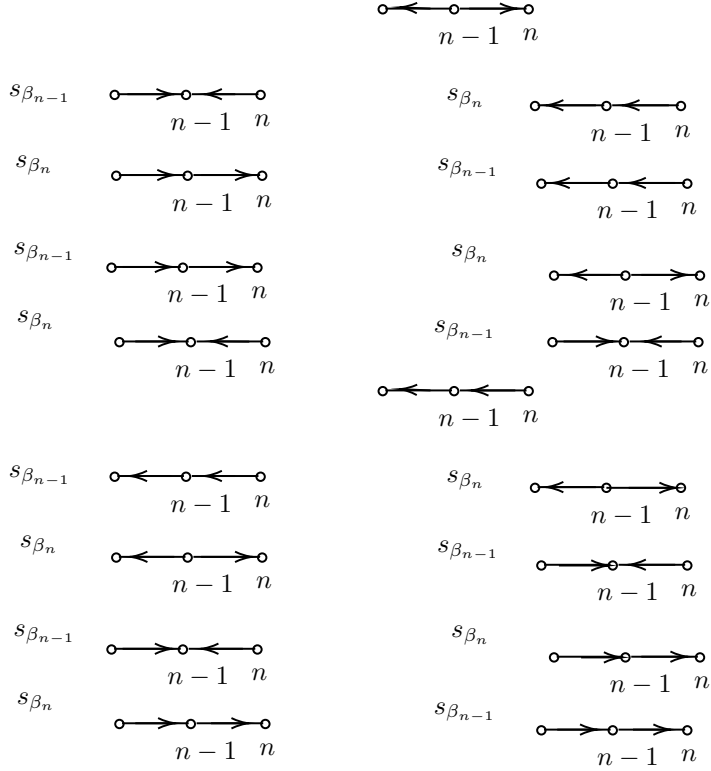
\tikzset{every picture/.style={line width=0.75pt}} 

%

(2) Since the only left descent of  $s_{\beta_n}s_{\beta_{n-1}}\cdots s_{\beta_1}s_{\delta-\rho}$ is $s_{\beta_n}$, the subgroup $\newline W_{\{s_{\beta_1}, \cdots, s_{\beta_{n-1}}\}}$ stabilizes $s_{\beta_n}s_{\beta_{n-1}}\cdots s_{\beta_1}s_{\delta-\rho}$.

Note that for a Coxeter element $c$ of $\widetilde{W}$, $s_{\delta-\rho}cs_{\delta-\rho}$ must be another Coxeter element.
We will next show that there must be some $w\in W$ such that $w\cdot c=s_{\delta-\rho}cs_{\delta-\rho}.$

To show this we prove a stronger fact. Let $O_c$ and $O_{c'}$ be two orientations that differ by the direction of one edge. We prove that one can find $w\in W$ such that $w\cdot O_c=O_{c'}$.
We argue by induction. Suppose $O_c$ and $O_{c'}$ differ by the first edge from the right. Then one can choose $w=s_{\beta_n}$.
Now we assume that $O_c$ and $O_{c'}$ differ by the $k-$th edge from the right.
Suppose that $i$ is the smallest number such that $n-k+1\leq i\leq n$ and $s_{\beta_i}\cdot c\neq c$.
Then $s_{\beta_{n-k+1}}\cdots s_{\beta_i}O_c$ is an orientation which differs from $O_{c'}$ by the direction of one edges with that differing edge closer to the right.
Therefore the assertion follows from induction.

By \cite{Shi2}, all Coxeter elements (for $\widetilde{W}$ of type $\widetilde{C}_n$) are conjugate. Combining this fact with the  discussion in the previous paragraph, one sees that $$W\cdot\{s_{\beta_n}s_{\beta_{n-1}}\cdots s_{\beta_1}s_{\delta-\rho}\}$$ is the full set of the Coxeter elements. This set has cardinality $2^{n}$. Then the cardinality of the stabilizer of $s_{\beta_n}s_{\beta_{n-1}}\cdots s_{\beta_1}s_{\delta-\rho}$ is $\frac{|W|}{2^n}=\frac{n!2^n}{2^n}=n!$. Therefore the stabilizer is exactly $W_{\{s_{\beta_1}, \cdots, s_{\beta_{n-1}}\}}$.
\end{proof}

\begin{lemma}\label{coxlemma}
(1) $w_{0}^{W_{S\backslash\{s_{\beta_n}\}}}w_0$ has a reduced expression $$(s_{\beta_n}s_{\beta_{n-1}}\cdots s_{\beta_1})(s_{\beta_n}s_{\beta_{n-1}}\cdots s_{\beta_2})\cdots (s_{\beta_n}s_{\beta_{n-1}})s_{\beta_n}$$

(2) $$t_{2\beta_1+4\beta_2+\cdots+2(n-1)\beta_{n-1}+n\beta_n}$$
$$=w_0^{W_{S\backslash \{s_{\beta_n}\}}}w_0s_{\delta-\rho}(s_{\beta_1}s_{\delta-\rho})(s_{\beta_2}s_{\beta_1}s_{\delta-\rho})\cdots(s_{\beta_{n-1}}s_{\beta_{n-2}}\cdots s_{\beta_1}s_{\delta-\rho})$$
$$=(s_{\beta_n}s_{\beta_{n-1}}\cdots s_{\beta_1}s_{\delta-\rho})^{n}$$




(3) The set of Coxeter elements of $\widetilde{W}$ equals to  $\{u^{-1}(s_{\beta_n}s_{\beta_{n-1}}\cdots s_{\beta_1}s_{\delta-\rho})u|$  $u$ is a left prefix of $w_0^{W_{S\backslash \{s_{\beta_n}\}}}w_0\}.$
\end{lemma}

\begin{proof}
One can verify (1) by checking that the actions on $\beta_1,\cdots, \beta_n$ of the left hand side and the right hand side coincide and  that the lengths of both sides agree. We omit the details.

To see (2), one computes
$$(s_{\beta_n}s_{\beta_{n-1}}\cdots s_{\beta_1})(s_{\beta_n}s_{\beta_{n-1}}\cdots s_{\beta_2})\cdots (s_{\beta_n}s_{\beta_{n-1}})s_{\beta_n}$$
$$s_{\delta-\rho}(s_{\beta_1}s_{\delta-\rho})(s_{\beta_2}s_{\beta_1}s_{\delta-\rho})\cdots(s_{\beta_{n-1}}s_{\beta_{n-2}}\cdots s_{\beta_1}s_{\delta-\rho})$$
$$=(s_{\beta_n}s_{\beta_{n-1}}\cdots s_{\beta_1}s_{\delta-\rho})$$
$$(s_{\beta_n}s_{\beta_{n-1}}\cdots s_{\beta_2})(s_{\beta_n}s_{\beta_{n-1}}\cdots s_{\beta_3})\cdots (s_{\beta_n}s_{\beta_{n-1}})s_{\beta_n}$$
$$(s_{\beta_1}s_{\delta-\rho})(s_{\beta_2}s_{\beta_1}s_{\delta-\rho})\cdots(s_{\beta_{n-1}}s_{\beta_{n-2}}\cdots s_{\beta_1}s_{\delta-\rho})$$
$$=(s_{\beta_n}s_{\beta_{n-1}}\cdots s_{\beta_1}s_{\delta-\rho})$$
$$(s_{\beta_n}s_{\beta_{n-1}}\cdots s_{\beta_2}s_{\beta_1})(s_{\beta_n}s_{\beta_{n-1}}\cdots s_{\beta_2})\cdots (s_{\beta_n}s_{\beta_{n-1}}s_{\beta_{n-2}})(s_{\beta_n}s_{\beta_{n-1}})$$
$$(s_{\delta-\rho})(s_{\beta_1}s_{\delta-\rho})\cdots(s_{\beta_{n-2}}\cdots s_{\beta_1}s_{\delta-\rho})$$
$$=(s_{\beta_n}s_{\beta_{n-1}}\cdots s_{\beta_1}s_{\delta-\rho})^2$$
$$(s_{\beta_n}s_{\beta_{n-1}}\cdots s_{\beta_2})\cdots (s_{\beta_n}s_{\beta_{n-1}}s_{\beta_{n-2}})(s_{\beta_n}s_{\beta_{n-1}})$$
$$(s_{\beta_1}s_{\delta-\rho})\cdots(s_{\beta_{n-2}}\cdots s_{\beta_1}s_{\delta-\rho})$$
$$=\cdots$$
$$=(s_{\beta_n}s_{\beta_{n-1}}\cdots s_{\beta_1}s_{\delta-\rho})^n.$$

(3) By Lemma \ref{actionB} (2), every Coxeter element of $\widetilde{W}$ is of the form $$u^{-1}\cdot s_{\beta_n}s_{\beta_{n-1}}\cdots s_{\beta_1}s_{\delta-\rho}$$ where $u$ is an element in $W$ whose left descent set is $\{s_{\beta_n}\}$, i.e. $u$ is a left prefix of $w_0^{W_{S\backslash \{s_{\beta_n}\}}}w_0.$
By comparison of the cardinality, the set of Coxeter elements and the set $$\{u^{-1}\cdot s_{\beta_n}s_{\beta_{n-1}}\cdots s_{\beta_1}s_{\delta-\rho}|\, u\, \text{is a left prefix of}\, w_0^{W_{S\backslash \{s_{\beta_n}\}}}w_0\}$$ have to be equal as there are $2^n$ left prefix of $w_0^{W_{S\backslash \{s_{\beta_n}\}}}w_0.$
This also forces that $u^{-1}\cdot s_{\beta_n}s_{\beta_{n-1}}\cdots s_{\beta_1}s_{\delta-\rho}=u^{-1}s_{\beta_n}s_{\beta_{n-1}}\cdots s_{\beta_1}s_{\delta-\rho}u$.
(Write $c=s_{\beta_n}s_{\beta_{n-1}}\cdots s_{\beta_1}s_{\delta-\rho}$. Suppose that $u=s_1s_2\cdots s_k$ and $k$ is  minimal such that $s_k\cdot (s_{k-1}\cdots s_1\cdot c)=(s_{k-1}\cdots s_1\cdot c)$, then the set $\{u^{-1}\cdot s_{\beta_n}s_{\beta_{n-1}}\cdots s_{\beta_1}s_{\delta-\rho}$ where $u$ is an element in $W$ whose left descent set  is $\{s_{\beta_n}\}\}$ will have fewer than $2^n$ elements.)

\end{proof}

\begin{theorem}
Every minimal infinite reduced word of the form
$$(\underline{v}s_{\delta-\rho}(s_{\beta_1}s_{\delta-\rho})(s_{\beta_2}s_{\beta_1}s_{\delta-\rho})\cdots(s_{\beta_{n-1}}s_{\beta_{n-2}}\cdots s_{\beta_1}s_{\delta-\rho})\underline{u})^{\infty}$$
where $\underline{uv}$ is a reduced expression of $w_0^{W_{S\backslash \{s_{\beta_n}\}}}w_0$ is an infinite Coxeter element. Conversely every infinite element has a reduced expression of such form.
\end{theorem}

\begin{proof}
The theorem follows from Lemma \ref{coxlemma} (1)-(3) and the computation
$$(\underline{v}s_{\delta-\rho}(s_{\beta_1}s_{\delta-\rho})(s_{\beta_2}s_{\beta_1}s_{\delta-\rho})\cdots(s_{\beta_{n-1}}s_{\beta_{n-2}}\cdots s_{\beta_1}s_{\delta-\rho})\underline{u})^{\infty}$$
$$=u^{-1}u(\underline{v}s_{\delta-\rho}(s_{\beta_1}s_{\delta-\rho})(s_{\beta_2}s_{\beta_1}s_{\delta-\rho})\cdots(s_{\beta_{n-1}}s_{\beta_{n-2}}\cdots s_{\beta_1}s_{\delta-\rho})\underline{u})^{\infty}$$
$$=(u^{-1}(s_{\beta_n}s_{\beta_{n-1}}\cdots s_{\beta_1}s_{\delta-\rho})^{n}u)^{\infty}$$
$$=(u^{-1}s_{\beta_n}s_{\beta_{n-1}}\cdots s_{\beta_1}s_{\delta-\rho}u)^{\infty}.$$
\end{proof}

\begin{remark*}
The above action may not be well-defined for other types. For example in $\widetilde{B}_4$,
$$s_{\beta_2}s_{\beta_3}s_{\beta_2}\cdot (s_{\beta_1}s_{\beta_4}s_{\beta_2}s_{\beta_3}s_{\delta-\rho})=s_{\beta_3}s_{\beta_1}s_{\beta_4}s_{\beta_2}s_{\delta-\rho}$$
but
$$s_{\beta_3}s_{\beta_2}s_{\beta_3}\cdot (s_{\beta_1}s_{\beta_4}s_{\beta_2}s_{\beta_3}s_{\delta-\rho})=s_{\beta_1}s_{\beta_4}s_{\beta_2}s_{\beta_3}s_{\delta-\rho}.$$

For $\widetilde{G_2}$, the following easy calculations describe which minimal infinite reduced words are infinite Coxeter elements.

$(s_{\beta_1}s_{\beta_2}s_{\delta-\rho})^2=s_{\beta_1}s_{\beta_2}s_{\beta_1}s_{\beta_2}s_{\delta-\rho}s_{\beta_2}$. Then set $\underline{v}=s_{\beta_1}s_{\beta_2}s_{\beta_1}s_{\beta_2}, \underline{u}=s_{\beta_2}$. One has $\underline{uv}=w_0^{W_{\{\beta_1\}}}w_0$. So
$(\underline{v}s_{\delta-\rho}\underline{u})^{\infty}=(s_{\beta_1}s_{\beta_2}s_{\delta-\rho})^{\infty}$.

$(s_{\beta_2}s_{\beta_1}s_{\delta-\rho})^2=s_{\beta_2}s_{\beta_1}s_{\beta_2}s_{\delta-\rho}s_{\beta_2}s_{\beta_1}$. Then set $\underline{v}=s_{\beta_2}s_{\beta_1}s_{\beta_2}, \underline{u}=s_{\beta_2}s_{\beta_1}$. One has $\underline{uv}=w_0^{W_{\{\beta_1\}}}w_0$. So
$(\underline{v}s_{\delta-\rho}\underline{u})^{\infty}=(s_{\beta_2}s_{\beta_1}s_{\delta-\rho})^{\infty}$.

$(s_{\beta_1}s_{\delta-\rho}s_{\beta_2})^2=s_{\beta_1}s_{\beta_2}s_{\delta-\rho}s_{\beta_2}s_{\beta_1}s_{\beta_2}$. Then set $\underline{v}=s_{\beta_1}s_{\beta_2}, \underline{u}=s_{\beta_2}s_{\beta_1}s_{\beta_2}$.  One has $\underline{uv}=w_0^{W_{\{\beta_1\}}}w_0$. So
$(\underline{v}s_{\delta-\rho}\underline{u})^{\infty}=(s_{\beta_1}s_{\delta-\rho}s_{\beta_2})^{\infty}$.

$(s_{\delta-\rho}s_{\beta_2}s_{\beta_1})^2=s_{\beta_2}s_{\delta-\rho}s_{\beta_2}s_{\beta_1}s_{\beta_2}s_{\beta_1}$. Then set $\underline{v}=s_{\beta_2}, \underline{u}=s_{\beta_2}s_{\beta_1}s_{\beta_2}s_{\beta_1}$. One has  $\underline{uv}=w_0^{W_{\{\beta_1\}}}w_0$. So
$(\underline{v}s_{\delta-\rho}\underline{u})^{\infty}=(s_{\delta-\rho}s_{\beta_2}s_{\beta_1})^{\infty}$.

From this example, $(s_{\beta_1}s_{\beta_2}s_{\delta-\rho})^{\infty}$ is a minimal infinite reduced word. Note that $s_{\beta_1}s_{\beta_2}s_{\delta-\rho}s_{\beta_1}s_{\beta_2}s_{\delta-\rho}$ is not fully commutative (a word is fully commutative if only trivial braid moves can be applied to it). So minimal infinite reduced word is not necessarily fully commutative. This answers another question in Section 10 of \cite{Lam2}.

For $\widetilde{F}_4$ the following calculations show that infinite Coxeter elements are all minimal infinite reduced words. Set
$$\underline{u}=s_{\beta_2}s_{\beta_3}s_{\beta_2}s_{\beta_4}s_{\beta_3}s_{\beta_1}s_{\beta_2}s_{\beta_3}s_{\beta_2}s_{\beta_4}s_{\beta_3},$$
$$\underline{v}=s_{\beta_2}s_{\beta_1}s_{\beta_2}s_{\beta_3}s_{\beta_4}s_{\beta_2}s_{\beta_3}s_{\beta_1}s_{\beta_2}.$$
Then $\underline{uv}=w_0^{W_{\{\beta_2\}}}w_0$. One has
$$\underline{v}s_{\delta-\rho}s_{\beta_1}s_{\beta_2}s_{\beta_3}s_{\beta_4}s_{\beta_2}s_{\beta_3}s_{\beta_2}s_{\beta_1}s_{\delta-\rho}\underline{u}$$
$$=(s_{\beta_1}s_{\beta_2}s_{\beta_3}s_{\beta_4}s_{\delta-\rho})^6.$$
Hence $(s_{\beta_1}s_{\beta_2}s_{\beta_3}s_{\beta_4}s_{\delta-\rho})^{\infty}=(\underline{v}s_{\delta-\rho}s_{\beta_1}s_{\beta_2}s_{\beta_3}s_{\beta_4}s_{\beta_2}s_{\beta_3}s_{\beta_2}s_{\beta_1}s_{\delta-\rho}\underline{u})^{\infty}$.

Set
$$\underline{u}=s_{\beta_2}s_{\beta_1}s_{\beta_3}s_{\beta_2}s_{\beta_3}s_{\beta_4}s_{\beta_3}s_{\beta_2}s_{\beta_3}s_{\beta_1}s_{\beta_2}s_{\beta_3}s_{\beta_2}s_{\beta_4}s_{\beta_3},$$
$$\underline{v}=s_{\beta_1}s_{\beta_2}s_{\beta_3}s_{\beta_1}s_{\beta_2}.$$
Then $\underline{uv}=w_0^{W_{\{\beta_2\}}}w_0$. One has
$$\underline{v}s_{\delta-\rho}s_{\beta_1}s_{\beta_2}s_{\beta_3}s_{\beta_4}s_{\beta_2}s_{\beta_3}s_{\beta_2}s_{\beta_1}s_{\delta-\rho}\underline{u}$$
$$=(s_{\delta-\rho}s_{\beta_1}s_{\beta_2}s_{\beta_3}s_{\beta_4})^6.$$
Hence $(s_{\delta-\rho}s_{\beta_1}s_{\beta_2}s_{\beta_3}s_{\beta_4})^{\infty}=(\underline{v}s_{\delta-\rho}s_{\beta_1}s_{\beta_2}s_{\beta_3}s_{\beta_4}s_{\beta_2}s_{\beta_3}s_{\beta_2}s_{\beta_1}s_{\delta-\rho}\underline{u})^{\infty}$.

Set
$$\underline{u}=s_{\beta_2}s_{\beta_3}s_{\beta_2}s_{\beta_4}s_{\beta_3}s_{\beta_1}s_{\beta_2}s_{\beta_3}s_{\beta_2}s_{\beta_4}s_{\beta_3}s_{\beta_1}s_{\beta_2}s_{\beta_3},$$
$$\underline{v}=s_{\beta_1}s_{\beta_4}s_{\beta_2}s_{\beta_3}s_{\beta_1}s_{\beta_2}.$$
Then $\underline{uv}=w_0^{W_{\{\beta_2\}}}w_0$. One has
$$\underline{v}s_{\delta-\rho}s_{\beta_1}s_{\beta_2}s_{\beta_3}s_{\beta_4}s_{\beta_2}s_{\beta_3}s_{\beta_2}s_{\beta_1}s_{\delta-\rho}\underline{u}$$
$$=(s_{\delta-\rho}s_{\beta_1}s_{\beta_2}s_{\beta_4}s_{\beta_3})^6.$$
Hence $(s_{\delta-\rho}s_{\beta_1}s_{\beta_2}s_{\beta_4}s_{\beta_3})^{\infty}=(\underline{v}s_{\delta-\rho}s_{\beta_1}s_{\beta_2}s_{\beta_3}s_{\beta_4}s_{\beta_2}s_{\beta_3}s_{\beta_2}s_{\beta_1}s_{\delta-\rho}\underline{u})^{\infty}$.

Set
$$\underline{u}=s_{\beta_2}s_{\beta_3}s_{\beta_2}s_{\beta_4}s_{\beta_3}s_{\beta_1}s_{\beta_2}s_{\beta_3}s_{\beta_2}s_{\beta_4}s_{\beta_3}s_{\beta_1}s_{\beta_2},$$
$$\underline{v}=s_{\beta_3}s_{\beta_1}s_{\beta_4}s_{\beta_2}s_{\beta_3}s_{\beta_1}s_{\beta_2}.$$
Then $\underline{uv}=w_0^{W_{\{\beta_2\}}}w_0$. One has
$$\underline{v}s_{\delta-\rho}s_{\beta_1}s_{\beta_2}s_{\beta_3}s_{\beta_4}s_{\beta_2}s_{\beta_3}s_{\beta_2}s_{\beta_1}s_{\delta-\rho}\underline{u}$$
$$=(s_{\beta_3}s_{\delta-\rho}s_{\beta_1}s_{\beta_2}s_{\beta_4})^6.$$
Hence $(s_{\beta_3}s_{\delta-\rho}s_{\beta_1}s_{\beta_2}s_{\beta_4})^{\infty}=(\underline{v}s_{\delta-\rho}s_{\beta_1}s_{\beta_2}s_{\beta_3}s_{\beta_4}s_{\beta_2}s_{\beta_3}s_{\beta_2}s_{\beta_1}s_{\delta-\rho}\underline{u})^{\infty}$.

Set
$$\underline{u}=s_{\beta_2}s_{\beta_1}s_{\beta_3}s_{\beta_2}s_{\beta_3}s_{\beta_1}s_{\beta_4}s_{\beta_3}s_{\beta_2}s_{\beta_3}s_{\beta_1}s_{\beta_2},$$
$$\underline{v}=s_{\beta_1}s_{\beta_4}s_{\beta_3}s_{\beta_4}s_{\beta_2}s_{\beta_3}s_{\beta_1}s_{\beta_2}.$$
Then $\underline{uv}=w_0^{W_{\{\beta_2\}}}w_0$. One has
$$\underline{v}s_{\delta-\rho}s_{\beta_1}s_{\beta_2}s_{\beta_3}s_{\beta_4}s_{\beta_2}s_{\beta_3}s_{\beta_2}s_{\beta_1}s_{\delta-\rho}\underline{u}$$
$$=(s_{\beta_4}s_{\beta_3}s_{\delta-\rho}s_{\beta_1}s_{\beta_2})^6.$$
Hence $(s_{\beta_4}s_{\beta_3}s_{\delta-\rho}s_{\beta_1}s_{\beta_2})^{\infty}=(\underline{v}s_{\delta-\rho}s_{\beta_1}s_{\beta_2}s_{\beta_3}s_{\beta_4}s_{\beta_2}s_{\beta_3}s_{\beta_2}s_{\beta_1}s_{\delta-\rho}\underline{u})^{\infty}$.

Set
$$\underline{u}=s_{\beta_2}s_{\beta_3}s_{\beta_2}s_{\beta_4}s_{\beta_3}s_{\beta_1}s_{\beta_2}s_{\beta_3}s_{\beta_2}s_{\beta_4}s_{\beta_3}s_{\beta_1},$$
$$\underline{v}=s_{\beta_2}s_{\beta_1}s_{\beta_3}s_{\beta_4}s_{\beta_2}s_{\beta_3}s_{\beta_1}s_{\beta_2}.$$
Then $\underline{uv}=w_0^{W_{\{\beta_2\}}}w_0$. One has
$$\underline{v}s_{\delta-\rho}s_{\beta_1}s_{\beta_2}s_{\beta_3}s_{\beta_4}s_{\beta_2}s_{\beta_3}s_{\beta_2}s_{\beta_1}s_{\delta-\rho}\underline{u}$$
$$=(s_{\beta_2}s_{\delta-\rho}s_{\beta_1}s_{\beta_3}s_{\beta_4})^6.$$
Hence $(s_{\beta_2}s_{\delta-\rho}s_{\beta_1}s_{\beta_3}s_{\beta_4})^{\infty}=(\underline{v}s_{\delta-\rho}s_{\beta_1}s_{\beta_2}s_{\beta_3}s_{\beta_4}s_{\beta_2}s_{\beta_3}s_{\beta_2}s_{\beta_1}s_{\delta-\rho}\underline{u})^{\infty}$.

Set
$$\underline{u}=s_{\beta_2}s_{\beta_1}s_{\beta_3}s_{\beta_2}s_{\beta_3}s_{\beta_1}s_{\beta_4}s_{\beta_3}s_{\beta_2}s_{\beta_3}s_{\beta_1},$$
$$\underline{v}=s_{\beta_4}s_{\beta_2}s_{\beta_1}s_{\beta_3}s_{\beta_4}s_{\beta_2}s_{\beta_3}s_{\beta_1}s_{\beta_2}.$$
Then $\underline{uv}=w_0^{W_{\{\beta_2\}}}w_0$. One has
$$\underline{v}s_{\delta-\rho}s_{\beta_1}s_{\beta_2}s_{\beta_3}s_{\beta_4}s_{\beta_2}s_{\beta_3}s_{\beta_2}s_{\beta_1}s_{\delta-\rho}\underline{u}$$
$$=(s_{\beta_4}s_{\delta-\rho}s_{\beta_2}s_{\beta_1}s_{\beta_3})^6.$$
Hence $(s_{\beta_4}s_{\delta-\rho}s_{\beta_2}s_{\beta_1}s_{\beta_3})^{\infty}=(\underline{v}s_{\delta-\rho}s_{\beta_1}s_{\beta_2}s_{\beta_3}s_{\beta_4}s_{\beta_2}s_{\beta_3}s_{\beta_2}s_{\beta_1}s_{\delta-\rho}\underline{u})^{\infty}$.

Set
$$\underline{u}=s_{\beta_2}s_{\beta_1}s_{\beta_3}s_{\beta_2}s_{\beta_3}s_{\beta_1}s_{\beta_4}s_{\beta_3}s_{\beta_2}s_{\beta_1},$$
$$\underline{v}=s_{\beta_3}s_{\beta_4}s_{\beta_2}s_{\beta_1}s_{\beta_3}s_{\beta_4}s_{\beta_2}s_{\beta_3}s_{\beta_1}s_{\beta_2}.$$
Then $\underline{uv}=w_0^{W_{\{\beta_2\}}}w_0$. One has
$$\underline{v}s_{\delta-\rho}s_{\beta_1}s_{\beta_2}s_{\beta_3}s_{\beta_4}s_{\beta_2}s_{\beta_3}s_{\beta_2}s_{\beta_1}s_{\delta-\rho}\underline{u}$$
$$=(s_{\beta_3}s_{\beta_4}s_{\delta-\rho}s_{\beta_2}s_{\beta_1})^6.$$
Hence $(s_{\beta_3}s_{\beta_4}s_{\delta-\rho}s_{\beta_2}s_{\beta_1})^{\infty}=(\underline{v}s_{\delta-\rho}s_{\beta_1}s_{\beta_2}s_{\beta_3}s_{\beta_4}s_{\beta_2}s_{\beta_3}s_{\beta_2}s_{\beta_1}s_{\delta-\rho}\underline{u})^{\infty}$.

Set
$$\underline{u}=s_{\beta_2}s_{\beta_3}s_{\beta_4}s_{\beta_2}s_{\beta_1}s_{\beta_2}s_{\beta_3}s_{\beta_2}s_{\beta_1},$$
$$\underline{v}=s_{\beta_4}s_{\beta_3}s_{\beta_4}s_{\beta_2}s_{\beta_1}s_{\beta_3}s_{\beta_4}s_{\beta_2}s_{\beta_3}s_{\beta_1}s_{\beta_2}.$$
Then $\underline{uv}=w_0^{W_{\{\beta_2\}}}w_0$. One has
$$\underline{v}s_{\delta-\rho}s_{\beta_1}s_{\beta_2}s_{\beta_3}s_{\beta_4}s_{\beta_2}s_{\beta_3}s_{\beta_2}s_{\beta_1}s_{\delta-\rho}\underline{u}$$
$$=(s_{\beta_4}s_{\beta_3}s_{\beta_2}s_{\delta-\rho}s_{\beta_1})^6.$$
Hence $(s_{\beta_4}s_{\beta_3}s_{\beta_2}s_{\delta-\rho}s_{\beta_1})^{\infty}=(\underline{v}s_{\delta-\rho}s_{\beta_1}s_{\beta_2}s_{\beta_3}s_{\beta_4}s_{\beta_2}s_{\beta_3}s_{\beta_2}s_{\beta_1}s_{\delta-\rho}\underline{u})^{\infty}$.

Set
$$\underline{u}=s_{\beta_2}s_{\beta_1}s_{\beta_3}s_{\beta_2}s_{\beta_3}s_{\beta_1}s_{\beta_4}s_{\beta_3}s_{\beta_2}s_{\beta_3},$$
$$\underline{v}=s_{\beta_2}s_{\beta_4}s_{\beta_1}s_{\beta_2}s_{\beta_3}s_{\beta_4}s_{\beta_2}s_{\beta_3}s_{\beta_1}s_{\beta_2}.$$
Then $\underline{uv}=w_0^{W_{\{\beta_2\}}}w_0$. One has
$$\underline{v}s_{\delta-\rho}s_{\beta_1}s_{\beta_2}s_{\beta_3}s_{\beta_4}s_{\beta_2}s_{\beta_3}s_{\beta_2}s_{\beta_1}s_{\delta-\rho}\underline{u}$$
$$=(s_{\beta_4}s_{\beta_1}s_{\beta_2}s_{\beta_3}s_{\delta-\rho})^6.$$
Hence $(s_{\beta_4}s_{\beta_1}s_{\beta_2}s_{\beta_3}s_{\delta-\rho})^{\infty}=(\underline{v}s_{\delta-\rho}s_{\beta_1}s_{\beta_2}s_{\beta_3}s_{\beta_4}s_{\beta_2}s_{\beta_3}s_{\beta_2}s_{\beta_1}s_{\delta-\rho}\underline{u})^{\infty}$.

Set
$$\underline{u}=s_{\beta_2}s_{\beta_1}s_{\beta_3}s_{\beta_2}s_{\beta_3}s_{\beta_1}s_{\beta_4}s_{\beta_3}s_{\beta_2},$$
$$\underline{v}=s_{\beta_3}s_{\beta_2}s_{\beta_4}s_{\beta_1}s_{\beta_2}s_{\beta_3}s_{\beta_4}s_{\beta_2}s_{\beta_3}s_{\beta_1}s_{\beta_2}.$$
Then $\underline{uv}=w_0^{W_{\{\beta_2\}}}w_0$. One has
$$\underline{v}s_{\delta-\rho}s_{\beta_1}s_{\beta_2}s_{\beta_3}s_{\beta_4}s_{\beta_2}s_{\beta_3}s_{\beta_2}s_{\beta_1}s_{\delta-\rho}\underline{u}$$
$$=(s_{\beta_3}s_{\beta_4}s_{\beta_1}s_{\beta_2}s_{\delta-\rho})^6.$$
Hence $(s_{\beta_3}s_{\beta_4}s_{\beta_1}s_{\beta_2}s_{\delta-\rho})^{\infty}=(\underline{v}s_{\delta-\rho}s_{\beta_1}s_{\beta_2}s_{\beta_3}s_{\beta_4}s_{\beta_2}s_{\beta_3}s_{\beta_2}s_{\beta_1}s_{\delta-\rho}\underline{u})^{\infty}$.

Set
$$\underline{u}=s_{\beta_2}s_{\beta_3}s_{\beta_4}s_{\beta_2}s_{\beta_1}s_{\beta_2}s_{\beta_3}s_{\beta_2},$$
$$\underline{v}=s_{\beta_4}s_{\beta_3}s_{\beta_2}s_{\beta_4}s_{\beta_1}s_{\beta_2}s_{\beta_3}s_{\beta_4}s_{\beta_2}s_{\beta_3}s_{\beta_1}s_{\beta_2}.$$
Then $\underline{uv}=w_0^{W_{\{\beta_2\}}}w_0$. One has
$$\underline{v}s_{\delta-\rho}s_{\beta_1}s_{\beta_2}s_{\beta_3}s_{\beta_4}s_{\beta_2}s_{\beta_3}s_{\beta_2}s_{\beta_1}s_{\delta-\rho}\underline{u}$$
$$=(s_{\beta_4}s_{\beta_3}s_{\beta_1}s_{\beta_2}s_{\delta-\rho})^6.$$
Hence $(s_{\beta_4}s_{\beta_3}s_{\beta_1}s_{\beta_2}s_{\delta-\rho})^{\infty}=(\underline{v}s_{\delta-\rho}s_{\beta_1}s_{\beta_2}s_{\beta_3}s_{\beta_4}s_{\beta_2}s_{\beta_3}s_{\beta_2}s_{\beta_1}s_{\delta-\rho}\underline{u})^{\infty}$.

Set
$$\underline{u}=s_{\beta_2}s_{\beta_1}s_{\beta_3}s_{\beta_2}s_{\beta_3}s_{\beta_1}s_{\beta_4}s_{\beta_3},$$
$$\underline{v}=s_{\beta_2}s_{\beta_3}s_{\beta_2}s_{\beta_4}s_{\beta_1}s_{\beta_2}s_{\beta_3}s_{\beta_4}s_{\beta_2}s_{\beta_3}s_{\beta_1}s_{\beta_2}.$$
Then $\underline{uv}=w_0^{W_{\{\beta_2\}}}w_0$. One has
$$\underline{v}s_{\delta-\rho}s_{\beta_1}s_{\beta_2}s_{\beta_3}s_{\beta_4}s_{\beta_2}s_{\beta_3}s_{\beta_2}s_{\beta_1}s_{\delta-\rho}\underline{u}$$
$$=(s_{\beta_2}s_{\beta_1}s_{\delta-\rho}s_{\beta_3}s_{\beta_4})^6.$$
Hence $(s_{\beta_2}s_{\beta_1}s_{\delta-\rho}s_{\beta_3}s_{\beta_4})^{\infty}=(\underline{v}s_{\delta-\rho}s_{\beta_1}s_{\beta_2}s_{\beta_3}s_{\beta_4}s_{\beta_2}s_{\beta_3}s_{\beta_2}s_{\beta_1}s_{\delta-\rho}\underline{u})^{\infty}$.

Set
$$\underline{u}=s_{\beta_2}s_{\beta_1}s_{\beta_3}s_{\beta_2}s_{\beta_1}s_{\beta_4}s_{\beta_3},$$
$$\underline{v}=s_{\beta_4}s_{\beta_2}s_{\beta_3}s_{\beta_2}s_{\beta_4}s_{\beta_1}s_{\beta_2}s_{\beta_3}s_{\beta_4}s_{\beta_2}s_{\beta_3}s_{\beta_1}s_{\beta_2}.$$
Then $\underline{uv}=w_0^{W_{\{\beta_2\}}}w_0$. One has
$$\underline{v}s_{\delta-\rho}s_{\beta_1}s_{\beta_2}s_{\beta_3}s_{\beta_4}s_{\beta_2}s_{\beta_3}s_{\beta_2}s_{\beta_1}s_{\delta-\rho}\underline{u}$$
$$=(s_{\beta_4}s_{\beta_2}s_{\beta_1}s_{\delta-\rho}s_{\beta_3})^6.$$
Hence $(s_{\beta_4}s_{\beta_2}s_{\beta_1}s_{\delta-\rho}s_{\beta_3})^{\infty}=(\underline{v}s_{\delta-\rho}s_{\beta_1}s_{\beta_2}s_{\beta_3}s_{\beta_4}s_{\beta_2}s_{\beta_3}s_{\beta_2}s_{\beta_1}s_{\delta-\rho}\underline{u})^{\infty}$.

Set
$$\underline{u}=s_{\beta_2}s_{\beta_1}s_{\beta_3}s_{\beta_2}s_{\beta_1}s_{\beta_4},$$
$$\underline{v}=s_{\beta_3}s_{\beta_4}s_{\beta_2}s_{\beta_3}s_{\beta_2}s_{\beta_4}s_{\beta_1}s_{\beta_2}s_{\beta_3}s_{\beta_4}s_{\beta_2}s_{\beta_3}s_{\beta_1}s_{\beta_2}.$$
Then $\underline{uv}=w_0^{W_{\{\beta_2\}}}w_0$. One has
$$\underline{v}s_{\delta-\rho}s_{\beta_1}s_{\beta_2}s_{\beta_3}s_{\beta_4}s_{\beta_2}s_{\beta_3}s_{\beta_2}s_{\beta_1}s_{\delta-\rho}\underline{u}$$
$$=(s_{\beta_3}s_{\beta_4}s_{\beta_2}s_{\beta_1}s_{\delta-\rho})^6.$$
Hence $(s_{\beta_3}s_{\beta_4}s_{\beta_2}s_{\beta_1}s_{\delta-\rho})^{\infty}=(\underline{v}s_{\delta-\rho}s_{\beta_1}s_{\beta_2}s_{\beta_3}s_{\beta_4}s_{\beta_2}s_{\beta_3}s_{\beta_2}s_{\beta_1}s_{\delta-\rho}\underline{u})^{\infty}$.

Set
$$\underline{u}=s_{\beta_2}s_{\beta_1}s_{\beta_3}s_{\beta_2}s_{\beta_1},$$
$$\underline{v}=s_{\beta_4}s_{\beta_3}s_{\beta_4}s_{\beta_2}s_{\beta_3}s_{\beta_2}s_{\beta_4}s_{\beta_1}s_{\beta_2}s_{\beta_3}s_{\beta_4}s_{\beta_2}s_{\beta_3}s_{\beta_1}s_{\beta_2}.$$
Then $\underline{uv}=w_0^{W_{\{\beta_2\}}}w_0$. One has
$$\underline{v}s_{\delta-\rho}s_{\beta_1}s_{\beta_2}s_{\beta_3}s_{\beta_4}s_{\beta_2}s_{\beta_3}s_{\beta_2}s_{\beta_1}s_{\delta-\rho}\underline{u}$$
$$=(s_{\beta_4}s_{\beta_3}s_{\beta_2}s_{\beta_1}s_{\delta-\rho})^6.$$
Hence $(s_{\beta_4}s_{\beta_3}s_{\beta_2}s_{\beta_1}s_{\delta-\rho})^{\infty}=(\underline{v}s_{\delta-\rho}s_{\beta_1}s_{\beta_2}s_{\beta_3}s_{\beta_4}s_{\beta_2}s_{\beta_3}s_{\beta_2}s_{\beta_1}s_{\delta-\rho}\underline{u})^{\infty}$.

\end{remark*}

\end{document}